\documentclass[11pt,leqno]{article}

\usepackage{amsmath,amsthm}

\usepackage{amssymb,latexsym}

\usepackage{enumerate}

\newcommand{\BE}{{\mathbb E}}

\newcommand{\BL}{{\mathbb L}}
\newcommand{\BN}{{\mathbb N}}
\newcommand{\BR}{{\mathbb R}}

\newcommand{\BP}{{\mathbb P}}

\newcommand{\VV}{{\mathcal{V}}}
\newcommand{\NN}{{\mathcal{N}}}

\newcommand{\WW}{{\mathcal W}}
\newcommand{\FF}{{\mathcal{F}}}
\newcommand{\GG}{{\mathcal{G}}}
\newcommand{\BB}{{\mathcal{B}}}
\newcommand{\HH}{{\mathcal H}}

\newcommand{\SM}{{\mathcal{S}}}

\newcommand{\MM}{{\mathcal{M}}}

\newcommand{\EE}{{\mathcal{E}}}

\newcommand{\BBM}{{\mathbf{M}}}
\newcommand{\BBX}{{\mathbf{X}}}

\newcommand{\tr}{\mbox{\rm Tr}}

\newtheorem{theorem}{\bf Theorem}[section]
\newtheorem{proposition}[theorem]{\bf Proposition}
\newtheorem{lemma}[theorem]{\bf Lemma}
\newtheorem{corollary}[theorem]{\bf Corollary}

\theoremstyle{definition}
\newtheorem*{definition}{Definition}

\newtheorem{example}[theorem]{Example}
\newtheorem{remark}[theorem]{\bf Remark}

\numberwithin{equation}{section}

\begin{document}
\title{Nonlinear parabolic SPDEs involving
Dirichlet operators}
\author {Tomasz Klimsiak and Andrzej Rozkosz\smallskip\\
}
\date{}
\maketitle
\begin{abstract}
We study the problem of existence, uniqueness and regularity of
probabilistic solutions of the Cauchy problem for nonlinear
stochastic partial differential equations involving operators
corresponding to regular (nonsymmetric) Dirichlet forms. In proofs
we combine the methods of backward doubly stochastic differential
equations with those of probabilistic potential theory and
Dirichlet forms.
\end{abstract}



\section{Introduction} \label{sec1}

In the present paper we are concerned with the  problem of
existence, uniqueness and regularity of  probabilistic solutions
of stochastic partial differential equations (SPDEs for short) of
the form
\begin{equation}
\label{eq0.0} du(t)=(A_tu+f(t,x,u))\,dt+\tilde g(t,x,u)\,dB_t,
\quad u(0)=\varphi.
\end{equation}
In (\ref{eq0.0}), $B$ is some $Q$-Wiener process and $A_t$,
$t\in[0,T]$, are operators associated with some family of regular
(nonsymmetric) Dirichlet forms satisfying mild regularity
assumptions. These assumptions are automatically satisfied if
$A_t=A$, $t\in[0,T]$. Therefore our results apply in particular to
equations (\ref{eq0.0}) with $A_t$ replaced by any operator $A$
corresponding to a regular Dirichlet form. The class of such
operators is quite wide. It contains both local operators, whose
model example is the Laplacian $\Delta$, and nonlocal operators,
whose model example is the fractional Laplacian
$\Delta^{\alpha/2}$ with $\alpha\in(0,2)$. Other interesting
examples are to be found for instance in
\cite{FOT,KR:JFA,KR:CM,MR}. An important example of the family of
operators depending on $t$ and satisfying our regularity
assumptions is the family of uniformly elliptic operators of the
form
\begin{equation}
\label{eq6.01} A_tu=\sum^d_{i,j=1}(a_{ij}(t,x)u_{x_i})_{x_j}
\end{equation}
with ellipticity constant not depending on $t$. Actually, in case
$A_t$ are of the form (\ref{eq6.01}),  we consider equations more
general then (\ref{eq0.0})  with coefficients $f,\tilde g$
depending on $u$ and its gradient $\nabla u$.

As for $\varphi,f,\tilde g$, we assume that
$\varphi,f(\cdot,\cdot,0),\tilde g(\cdot,\cdot,0)$ are
square-integrable, $f(t,x,\cdot)$ is continuous and monotone (no
assumption on the growth of $f(t,x,\cdot)$ is imposed) and $\tilde
g(t,x,\cdot)$ is Lipschitz -continuous. In the case where
$f,\tilde g$ depend on $u$ and $\nabla u$, we also assume that
they are Lipschitz-continuous with respect to $\nabla u$.

To study (\ref{eq0.0}) we develop  the approach used successfully
in \cite{K:JFA,KR:JFA,KR:CM} to investigate Sobolev space
solutions of semilinear PDEs with operators corresponding to
Dirichlet forms. In those papers  PDEs are studied by the methods
of the theory of backward  stochastic differential equations
(BSDEs for short) combined with those of the probabilistic
potential theory and Dirichlet forms. In the present paper the
strategy for studying SPDEs is similar. The major difference is
that now we use backward doubly stochastic differential equations
(BDSDEs for short) instead of BSDEs. The idea of studying
nonlinear SPDEs via BDSDEs goes back to \cite{PP}. In \cite{PP}
classical solutions of equations with nondivergence form operators
with regular coefficients are considered. Our approach to
(\ref{eq0.0}) was also motivated by the desire to develop the
ideas of \cite{PP} to encompass a broader class of operators and
to study Sobolev space solutions.

As a matter of fact, we study the following Cauchy problem with
terminal condition:
\begin{equation}
\label{eq0.1}
du(t)=-(A_tu+f(t,x,u))\,dt-g(t,x,u)\,d^{\dagger}\beta_t, \quad
u(T)=\varphi.
\end{equation}
Here $g=(g_k)$ is a sequence of real functions  on
$\Omega\times(0,T]\times E\times\BR$ determined by $\tilde g$ and
$Q$, $\beta=(\beta_k)$ is a sequence of one-dimensional mutually
independent standard Wiener processes defined on some probability
space $(\Omega,\FF,P)$ and
$g\,d^{\dagger}\beta_t=\sum^{\infty}_{k=1}g_k\,d^{\dagger}\beta^k_t$,
where $g_k\,d^{\dagger}\beta^k_t$ denotes the backward It\^o
integral. The results for (\ref{eq0.1}) can be easily translated
into results for
(\ref{eq0.0}). 
However, since we heavily rely on the theory of BDSDEs, problem
(\ref{eq0.1}) is much more convenient to deal with.

Roughly speaking, our strategy for investigating (\ref{eq0.1})
consists of two steps. Suppose that the operators $A_t$ are
associated with some family of Dirichlet forms $\{B^{(t)}\}$ on
$L^2(E;m)$ with common domain $V$ and let $\EE$ be the
time-dependent Dirichlet form determined by $\{B^{(t)}\}$. Denote
by $\BBM=\{(\BBX,P_{z});z\in\BR\times E\}$ a time-space Markov
process with life time $\zeta$ associated with $\EE$. In the first
step we prove that there exists an exceptional set $N\subset
E_{0,T}:=(0,T]\times E$ and a pair of processes $(Y,M)$ such that
for every $z\in E_{0,T}\setminus N$ the process $M$ is a
martingale under $P\otimes P_z$ and $(Y,M)$ is a unique solution
of the BDSDE
\begin{align}
\label{eq1.02} Y_t&=\varphi(\BBX_{T_{\iota}}) +\int^{\zeta\wedge
T_{\iota}}_tf(\BBX_r,Y_r)\,dr+\int^{\zeta\wedge
T_{\iota}}_tg(\BBX_r,Y_r)\,d^{\dagger}\beta^{\iota}_r\\
&\quad -\int^{\zeta\wedge T_{\iota}}_tdM_r,\quad
t\in[0,\zeta\wedge T_{\iota}], \quad P\otimes P_z\mbox{-a.s.,}
\nonumber
\end{align}
where $T_{\iota}=T-\iota(0)$. Here $\iota$ is the uniform motion
to the right (in particular, $\iota(0)=s$ under $P_z$ with
$z=(s,x)$) and $\beta^{\iota}_t=\beta_{t+\iota(0)}$, $t\ge0$. In
fact, we show that
\[
Y_t=u(\BBX_t),\quad t\in [0,\zeta\wedge T_{\iota}]
\]
for some $u:\Omega\times E_{0,T}\rightarrow\BR$ such that
$u(\omega;\cdot)$ is quasi-continuous (with respect to the
capacity associated with $\EE$). Therefore putting $t=0$ in
(\ref{eq1.02}) and taking the expectation with respect to $P_z$ we
see that $u$ satisfies the nonlinear Feynman-Kac formula
\begin{align}
\label{eq1.03}
u(z)=E_{z}\Big(\varphi(\BBX_{T_{\iota}})
+\int^{\zeta\wedge T_{\iota}}_0f(\BBX_t,u(\BBX_t))\,dt
+\int^{\zeta\wedge T_{\iota}}_0
g(\BBX_t,u(\BBX_t))\,d^{\dagger}\beta^{\iota}_t\Big).
\end{align}
In fact, a quasi-continuous $u$ satisfying (\ref{eq1.03}) is
unique and we call it the (probabilistic) solution of
(\ref{eq0.0}). The second step is to use (\ref{eq1.03})  to derive
regularity properties of $u$. Our main result says that there is
$c>0$ such that
\begin{align}
\label{eq1.9} &\sup_{0\le t\le T}
E\|u(t)\|^2_{L^2(E;m)}+E\int^T_0\|u(t)\|^2_V\,dt\\
&\qquad\le cE\Big(\|\varphi\|^2_{L^2(E;m)}+
\int^T_0(\|f(t,\cdot,0)\|^2_{L^2(E;m)}+\sum^{\infty}_{k=1}
\|g_k(t,\cdot,0)\|^2_{L^2(E;m)})\,dt\Big). \nonumber
\end{align}

The problems of existence and uniqueness of solutions of
(\ref{eq0.0}) have been the subject of numerous papers. Here let
us mention important papers \cite{PZ1,PZ2} dealing with mild
solutions of equations of the form (\ref{eq0.0}) with $A_t=\Delta$
and $B$ being a spatially homogeneous Wiener process. We also
refer to \cite{PZ1,PZ2} for many bibliographic comments on
(\ref{eq0.0}) with $A_t=\Delta$. Note that in \cite{PZ1,PZ2} the
``semigroup approach" to (\ref{eq0.0}) is used. For results on
(\ref{eq0.0}) (with $A_t$ being local operators) which can be
obtained by using the ``variational approach", see \cite{Kr,R} and
the references therein. Solutions in the sense defined by Walsh
\cite{W} are considered for instance in \cite{BGP,GP} (see also
the expository paper \cite{DQS}).

The approach of \cite{PP}, via BDSDE, was developed in several
papers. In \cite{BM,BCM,WZ,ZZ} it is assumed that the operators
$A_t$ are the same as in \cite{PP}, i.e. second order operators in
nondivergence form with coefficients having some regularity
properties. In \cite{BM,ZZ} under the assumption that $f,\tilde g$
are Lipschitz continuous a stochastic representation of weak
solutions of the Cauchy problem for SPDEs in terms of BDSDE is
given. In \cite{WZ} the assumption that $f$ is Lipschitz
continuous in $u$ is weakened to monotonicity combined with the
linear growth condition in $u$. In \cite{BCM} a stochastic
representation is given for SPDEs with nonlinear boundary Neumann
conditions. Paper \cite{ZZ} also deals with stationary solutions
of SPDEs and related BDSDEs with infinite horizon. In \cite{DS}
the semigroup method is used to prove that if  $f,g$ are Lipschitz
continuous then there exists a unique mild solution of SPDE
involving general non-negative self-adjoint operator $A$ (not
depending on $t$) and finite-dimensional noise. Then this
analytical result is used to get a stochastic representation of
the solution in case $A$ is a diffusion operator. It is also worth
noting that in \cite{BGP,GP}, in case $B$ is a space-time white
noise, $A_t=\partial^2/\partial x^2$ and $g$ is nondegenerate,
existence and uniqueness results for (\ref{eq0.0}) with irregular
$f$ are proved ($f$ is merely measurable and satisfies some
integrability condition in case $g$ is constant, or $f$ is
measurable and locally bounded and $g$ is nondegenerate and
satisfies some regularity conditions).

The novelty of our paper lies in the fact that we prove in a
unified way the existence and uniqueness of solutions of
(\ref{eq0.0}) for much wider classes of operators then those
considered in previous papers or under less restrictive
assumptions on $f$ (see, however, the one dimensional results
proved in \cite{BGP,GP}). In particular, in contrast to \cite{DS},
in our paper the operators $A_t$  may depend on time and we only
assume that $f$ is continuous and monotone in $u$ (no restrictions
on the growth of $f$ with respect to $u$). Secondly, we show in a
unified way that the solutions of (\ref{eq0.0}) belong to some
Sobolev spaces and we prove energy estimates. Moreover, we obtain
stochastic representation of solutions of the form (\ref{eq1.03})
for quasi-every (and not just $m$-a.e.) point in $E_{0,T}$. Let us
stress one again, however, that except for some special cases, we
have to assume that $B$ is a $Q$-Wiener process with $Q$ of trace
class. Therefore our result do not cover the existence and
uniqueness results of \cite{PZ1,PZ2} obtained for SPDEs with a
spatially homogeneous Wiener process.

Our methods of proofs are also new. We think that of particular
interest is our method for deriving from (\ref{eq1.03}) regularity
properties of $u$. The method is probabilistic in nature. As
already mentioned, the general idea comes from our previous papers
\cite{K:JFA,KR:JFA,KR:CM} on PDEs and BSDEs. Here, however, new
difficulties and subtleties arise. The idea is as follows.  We
show that given a solution $u$ of (\ref{eq0.0}), i.e. a
quasi-continuous function  $u$ satisfying (\ref{eq1.03}), one can
find a process  $M$ such that $(Y,M)=(u(\BBX),M)$ is a solution of
BDSDE (\ref{eq1.02}). Thus, in fact, (\ref{eq1.02}) and
(\ref{eq1.03}) are equivalent. It is also worth mentioning here
that the regularity of trajectories of the process $u(X)$ (i.e.
the fact that $u(X)$ is  c\`adl\`ag and $[u(X)]_{-}=u(X_{-})$) do
not follow directly from the deterministic potential theory,
because in general the nest for the quasi-continuous function
$u(\omega;\cdot)$ depends on $\omega\in\Omega$. From
(\ref{eq1.02}) we immediately get
\begin{equation}
\label{eq1.05} A^{[u]}_t:=u(\BBX_t)-u(\BBX_0)=M_t+N_t,\quad
t\in[0,T_{\iota}],
\end{equation}
with
\begin{equation}
\label{eq1.06} N_t=-\int^t_0f(\BBX_r,u(\BBX_r))\,dr-
\int^t_0g(\BBX_r,u(\BBX_r))\,d^{\dagger}\beta^{\iota}_r.
\end{equation}
The process $A^{[u]}$ in (\ref{eq1.05}) is a random additive
functional (AF) of the part $\BBM^{0,T}$ of the process $\BBM$ on
$E_{0,T}$. We show that $M$ is a random martingale AF of
$\BBM^{0,T}$ of finite energy and $N$ is a random continuous AF of
$\BBM^{0,T}$  of finite energy (we introduce these notions in
Section \ref{sec4}). However, in most interesting cases $N$ is not
of zero energy, because from (\ref{eq1.06}) it follows that
\[
Ee(N)=\frac12 E\int^T_0\!\!\int_E\|g(t,x,u(t,x))\|^2\,dt\,m(dx).
\]
Therefore (\ref{eq1.05}) cannot be viewed  as Fukushima's
decomposition of $A^{[u]}$. Nevertheless, we are able to prove the
following formula for the energy of $M$:
\begin{equation}
\label{eq1.07} Ee(M)= E\Big(\|u(0)\|^2_{L^2(E;m)} +\int_0^T
B^{(t)}(u(t),u(t))\,dt -\frac12\int_{E_{0,T}}|u(z)|^2\,k(dz)\Big),
\end{equation}
where $k$ is some killing measure. Roughly speaking, we obtain
energy estimate (\ref{eq1.9}) for $u$ by combining a priori
estimates for the solution $(u(\BBX),M)$ of (\ref{eq1.02}) with
the estimate (\ref{eq1.07}). The estimates for $(u(\BBX),M)$ are
proved by using the methods of BSDEs. We also prove that if
$Ee(N)>0$ then
\begin{equation}
\label{eq1.10} u\notin\mathbb{W}=\{v\in L^2(\Omega\times(0,T);V);
\frac{\partial v}{\partial t}\in L^2(0,T; V')\,\,P\mbox{-a.s.}\},
\end{equation}
which shows the difference between the regularity theory for
(\ref{eq0.0}) and for usual PDEs.

In the last section of the paper we show a connection between
probabilistic and mild solutions of (\ref{eq0.1}) in case $f$ is
Lipschitz continuous in $u$ and $A_t=A$, $t\in[0,T]$. Roughly
speaking, changing the order of integration in (\ref{eq1.03}) and
using the fact that
\[
E_x f(X_t)=P_tf(x),\quad f\in L^2(E;m),\quad m\mbox{-a.e. }x\in E,
\]
where $\{P_t,\, t\ge 0\}$ is the semigroup on $L^2(E; m)$
generated by $A$,  we get after some direct calculation that
\[
u(s)=P_{T-s}\varphi+\int^{T}_sP_{t-s}F(t,u(t))\,dt
+\int^{T}_sP_{t-s}G(t,u(t))\,d^{\dagger}B_t
\]
where $F$ and $G$ are the Nemitskii operators corresponding to $f$
and $\tilde g$.

Finally, note that unlike \cite{K:JFA,KR:JFA,KR:CM}, in the case
where $A_t$ are defined by (\ref{eq6.01}), in the paper we treat
regularity of equations with coefficients $f,g$ depending both on
the solution and its  gradient.

\section{General BDSDEs}
\label{sec2}

In this section we consider general (non-Markovian) BDSDEs with
final condition $\xi$ and coefficients $f,g$ (BDSDE$(\xi,f,g)$ for
short) of the form
\begin{equation}
\label{eq3.1} Y_t=\xi+\int_t^Tf(r,Y_r)\,dr+\int_t^Tg(r,Y_r)\,
d^\dagger\beta_r-\int_t^T\,dM_r,\quad  t\in[0,T].
\end{equation}
To formulate the definition of a solution of (\ref{eq3.1}) we need
some notation. In what follows $\beta=(\beta^k)_{k\in\BN}$ is a
sequence of mutually independent one-dimensional standard Wiener
processes defined on some probability space $(\Omega,\FF,P)$ and
$(\Omega',\GG,(\GG_t)_{t\in[0,T]},P')$ is some filtered
probability space such that $(\GG_t)$ is right continuous and
complete. We set
\[
\FF_{t,T}^{\beta}=\sigma(\beta^k_r-\beta^k_t,\,r\in[t,T],k\in\BN)
\vee\NN,\quad \FF^{\beta}_t=\FF^{\beta}_{0,t},\quad t\in[0,T],
\]
where $\NN=\{A\subset\Omega:\exists B\in\FF^{\beta}_T\mbox{ such
that }A\subset B, P(B)=0\}$, and then we set
\[
\FF_t=\FF_{t,T}^{\beta}\vee\GG_t, \quad t\in[0,T].
\]
Note that $(\FF_t)$ is not increasing, so it is not  filtration.
We also set
\[
\mathbb{P}=P\otimes P'
\]
and by $\BE$ (resp. $E,E'$) we denote the expectation with respect
to the measure $\BP$ (resp. $P,P'$). Let  $X$ be a process defined
on $\Omega$, and $Y$ be a process on $\Omega'$. Throughout what
follows, without explicit mention we shall freely identify them
with  processes on $\Omega\times\Omega'$ defined as
\begin{equation}
\label{eq2.02} X(\omega,\omega')=X(\omega),\quad
Y(\omega,\omega')=Y(\omega')
\end{equation}
We will need  the following spaces.

\begin{itemize}
\item $M$ is the space of measurable processes $X=\{(X_t)_{t\in[0,T]}\}$
defined on $(\Omega\times\Omega',\FF\otimes\GG,\mathbb{P})$ such
that for a.e. $t\in [0,T]$ the random variable $X_t$ is
$\FF_t$-measurable. $M^2$ is the subspace of $M$ consisting of all
processes $X$ such that $(\BE\int_0^T|X_t|^2\,dt)^{1/2}<\infty$.

\item $\SM^2$ is the space of
c\`adl\`ag processes $Y\in M$  such that
$\|Y\|^2_{\SM^2}=\BE\sup_{0\le t\le T}|Y_t|^2<\infty$.

\item  $\MM^2$ is the space of
c\`adl\`ag processes $X\in M$  such that  $X$ is an
$(\FF^\beta_{T}\vee\GG_t)$-martingale  such that $X_0=0$ and
$\BE[X]_T<\infty$, where $[X]$ denotes the quadratic variation
process of $X$.
\end{itemize}

We will look for solutions of (\ref{eq3.1}) in the space
$\SM^2\times\MM^2$. Note that $\SM^2$ equipped with the norm
$\|\cdot\|_{\SM^2}$ is a Banach space. Similarly, $\MM^2$ equipped
with the norm $\|M\|_{\MM^2}=(\BE[M]_T)^{1/2}$ is a Banach space.

\begin{remark}
By a standard argument, if $M\in\MM^2$ then for $P$-a.e.
$\omega\in \Omega$ the process $M(\omega,\cdot)$ is a square
integrable $(\GG_t)$- martingale.
\end{remark}

Assume that we are given an $\FF_T$\,-measurable random variable
$\xi$ and two families $\{f(t,y),t\ge0\}_{y\in\BR}$,
$\{g_k(t,y),t\ge0\}_{y\in\BR,k\in\BN}$ of processes
$f(\cdot,y),g_k(\cdot,y):\Omega\times\Omega'\times[0,T]\rightarrow\BR$
of class $M$ (for brevity, in our notation we omit the dependence
on $(\omega,\omega')\in\Omega\times\Omega'$). Let
$g(\cdot,y)=(g_1(\cdot,y),g_2(\cdot,y),\dots)$ and let
$\|x\|=(\sum^{\infty}_{k=1}|x_k|^2)^{1/2}$ denote the usual norm
in the space $l^2$.

\begin{definition}
We say that a pair $(Y,M)\in\SM^2\times\MM^2$  is a solution of
BDSDE$(\xi,f,g)$ if
\begin{enumerate}
\item[(a)]$\int_0^T|f(t,Y_t)|\,dt<\infty$,
$\int_0^T\|g(t,Y_t)\|^2\,dt<\infty$ $\BP$-a.s.
\item[(b)] Eq. (\ref{eq3.1}) is satisfied $\BP$-a.s.
(In (\ref{eq3.1}), $\int^T_tg(r,Y_r)\,d^{\dagger}\beta_r
=\sum^{\infty}_{k=1}\int^T_tg_k(r,Y_r)\,d^{\dagger}\beta^k_r$ and
the integrals  involving the processes $\beta^k$ are backward
It\^o's integrals; note that under (a) the series converges in
$L^2(\Omega,\FF,P)$ for $P'$-a.s. $\omega'\in\Omega'$).
\end{enumerate}
\end{definition}

\begin{remark}
\label{rem2.2} Set
\[
\hat\beta=(\hat\beta^k)_{k\in\BN},\quad
\hat\beta^k_t=\beta^k_{T-t}-\beta^k_T,\quad t\in[0,T],
\]
and define $\FF_{t,T}^{\hat\beta}$, $\FF^{\hat\beta}_t$ to be
$\FF_{t,T}^{\beta}$, $\FF^{\beta}_t$ but with $\beta$ replaced by
$\hat\beta$. Then $\hat{\beta}$ is a sequence of mutually
independent standard Wiener processes.  If $\eta_t$ is
$\FF^{\beta}_{t,T}$-measurable, then $\eta_{T-t}$ is
$\FF^{\hat\beta}_t$-measurable and one can check that if
$\eta=(\eta^1,\eta^2,\dots)$ is an
$(\FF^{\beta}_{t,T})_{t\in[0,T]}$-adapted process such that
$P(\int^T_0\|\eta_t\|^2\,dt<\infty)=1$ then
\[
\int^T_t\eta_s\,d^{\dag}\beta_s
=-\int^{T-t}_0\eta_{T-s}\,d\hat\beta_s,\quad t\in[0,T]
\]
(see \cite[p. 176]{ZZ}).
\end{remark}

We are going to show that there exists a unique solution of
(\ref{eq3.1}) under the following assumptions.
\begin{enumerate}
\item[(A1)] $\BE|\xi|^2<\infty$, $\BE\int_0^T|f(t,0)|^2\,dt
+\BE\int_0^T\|g(t,0)\|^2\,dt<\infty$.
\item[(A2)] For every $y\in\BR$, $\int_0^T|f(t,y)|\,dt<\infty$
$\mathbb{P}$-a.s.
\item[(A3)]There exist constants $l,L>0$ and functions
$L_k:\Omega\times[0,T]\rightarrow\mathbb{R}_+$ such that
$\sup_{t\le T}\sum_{k=1}^{\infty}L^2_k(t)\le l$ $P$-a.s. and for
a.e. $t\in[0,T]$,
\begin{enumerate}
\item[(a)] $(f(t,y)-f(t,y'))(y-y')\le L|y-y'|^2$ for every
$y,y'\in\BR$, $\BP$-a.s.,
\item[(b)] $|g_k(t,y)-g_k(t,y')|\le L_k(t)|y-y'|$ for for every
$y,y'\in\BR$, $\BP$-a.s.
\end{enumerate}
\item[(A4)] For a.e. $t\in[0,T]$ the mapping $\BR\ni y\mapsto f(t,y)$
is continuous.
\end{enumerate}

The uniqueness for (\ref{eq3.1}) follows from the following
comparison result.

\begin{proposition}
\label{stw1.1} Let $g$  satisfy \mbox{\rm(A3b)} and either $f$ or
$f'$ satisfy $\mbox{\rm{(A3a)}}$. Let $(Y,M)$  be a solution of
\mbox{\rm BDSDE}$(\xi,f,g)$ and $(Y',M')$ be a solution of
\mbox{\rm BSDE}$(\xi',f',g)$. If $\xi\le\xi'$ $\mathbb{P}$-a.s.
and $f'(t,y)\le f(t,y)$ for a.e. $t\in[0,T]$ and every $y\in\BR$
then
\[
Y'_t\le Y_t,\quad t\in[0,T],\quad\BP\mbox{{-a.s.}}
\]
\end{proposition}
\begin{proof}
We assume that $f$ satisfies (A2). In case $f'$ satisfies (A2) the
proof is analogous. By the It\^o-Meyer formula,
\begin{align*}
&|(Y_t'-Y_t)^+|^2+\int_t^T\mathbf{1}_{\{Y'_r>Y_r\}}\,d[M'-M]_r\le
2\int_t^T(Y'_r-Y_r)^+(f'(r,Y_r')-f(r,Y_r))\,dr\\
&\qquad+2\int_t^T(Y'_{r-}-Y_{r-})^+
(g(r,Y'_r)-g(r,Y_r))\,d^{\dagger}\beta_r
-2\int_t^T(Y'_{r-}-Y_{r-})^+\,d(M'-M)_r
\\&\qquad+\sum^{\infty}_{k=1}\int_t^T
\mathbf{1}_{\{Y'_r>Y_r\}}|g_k(r,Y'_r)-g_k(r,Y_r)|^2\,dr.
\end{align*}
By the assumptions,
\[
\int_t^T(Y'_r-Y_r)^+(f'(r,Y'_r)-f(r,Y_r))\,dr
\le L\int_t^T|(Y'_r-Y_r)^+|^2\,dr
\]
and
\begin{align*}
\sum^{\infty}_{k=1}\int_t^T
\mathbf{1}_{\{Y'_r>Y_r\}}|g_k(r,Y'_r)-g_k(r,Y_r)|^2\,dr
&\le
\sum^{\infty}_{k=1}\int_t^TL^2_k(r)
\mathbf{1}_{\{Y'_r>Y_r\}}|Y'_r-Y_r|^2\,dr\\
&\le l\int_t^T|(Y'_r-Y_r)^+|^2\,dr.
\end{align*}
Hence
\[
\mathbb{E}|(Y'_t-Y_t)^+|^2\le2(L+l)\mathbb{E}\int_t^T|(Y'_r-Y_r)^+|^2\,dr,
\quad t\in[0,T],
\]
so applying Gronwall's lemma we get the desired result.
\end{proof}

\begin{corollary} \label{wn1.1} Let assumption
\mbox{\rm(A3)} hold. Then there exists at most one solution
of \mbox{\rm BDSDE}$(\xi,f,g)$.
\end{corollary}

\begin{proposition}
\label{stw1.2} Assume $\mbox{\rm{(A1), (A3)}}$. Let $(Y,M)$ be a
solution of $\mbox{\rm{BDSDE}}(\xi,f,g)$. Then there exists $c>0$
depending only on $T,l,L$ such that
\begin{align*}
\BE\Big(\sup_{t\le T}|Y_t|^2+\int_0^T\,d[M]_t
&+\sup_{0\le t\le T}
|\int_0^t f(r,Y_r)\,dr|^2\Big)\\
&\quad\le c\BE\Big(|\xi|^2+\int_0^T(|f(t,0)|^2
+\|g(t,0)\|^2)\,dt\Big).
\end{align*}
\end{proposition}
\begin{proof}
By the It\^o-Meyer formula and (A3),
\begin{align*}
|Y_t|^2+\int_t^T\,d[M]_r&=|\xi|^2+2\int_t^Tf(r,Y_r)Y_r\,dr
+\int_t^T\|g(r,Y_r)\|^2\,dr\\
&\quad+2\int_t^TY_rg(r,Y_r)\,d^\dagger\beta_r
-2\int_t^TY_{r-}\,dM_r\\
&\le|\xi|^2+(1+2L+2l)\int_t^T|Y_r|^2\,dr
+\int_t^T(|f(r,0)|^2+2\|g(r,0)\|^2)\,dr \\
&\quad+2\int_t^TY_rg(r,Y_r)\,d^\dagger \beta_r
-2\int_t^TY_{r-}\,dM_r,\quad t\in[0,T],\quad \BP\mbox{-a.s.}
\end{align*}
From this and Gronwall's lemma,
\[
\BE|Y_t|^2+\BE\int_0^T\,d[M]_t\le c(L,l,T)\BE\Big(|\xi|^2
+\int_0^T(|f(t,0)|^2+\|g(t,0)\|^2)\,dt\Big).
\]
Applying the Burkholder-Davis-Gundy inequality we get the desired
result.
\end{proof}

\begin{theorem}
\label{tw1.1} Let assumptions \mbox{\rm(A1)--(A4)} hold. Then
there exists a solution $(Y,M)$ of \mbox{\rm BDSDE}$(\xi,f,g)$.
\end{theorem}
\begin{proof}
Step 1. Assume that $g(t,Y_t)=g(t)$, $f(t,Y_t)=f(t)$, $t\in[0,T]$.
We first prove that there exists a solution $(Y,M)$ of the linear
equation
\begin{equation}
\label{eq3.3} Y_t=\xi+\int_t^Tf(r)\,dr+\int_t^Tg(r)\,d^\dagger
\beta_r-\int_t^T\,dM_r, \quad t\in[0,T],\quad\BP\mbox{-a.s.}
\end{equation}
To this end, let us set $\mathcal{A}_t=\FF_T^{\beta}\vee\GG_t$,
$t\in[0,T]$, and define $Y,M$ as
\[
Y_t=\BE\Big(\xi+\int_t^Tf(r)\,dr +\int_t^Tg(r)\,d^\dagger
\beta_r|\mathcal{A}_t\Big),\quad t\in[0,T]
\]
and
\begin{align*}
M_t&=\mathbb{E}\Big(\xi+\int_0^Tf(r)\,dr+\int_0^Tg(r)\,d\beta_r^\dagger
|\mathcal{A}_t\Big)-Y_0\\
&=E\Big(\xi+\int_0^Tf(r)\,dr
+\int_0^Tg(r)\,d\beta_r^\dagger|\mathcal{G}_t\Big)-Y_0,
\quad t\in[0,T].
\end{align*}
One can check that the pair $(Y,M)$  satisfies (\ref{eq3.3}). To
show that $Y,M$ are $(\FF_t)$-adapted, let us set
\begin{equation}
\label{eq00.00}
\Lambda_t=\xi+\int_t^Tf(r)\,dr+\int_t^Tg(r)\,d^\dagger
\beta_r,\quad t\in[0,T].
\end{equation}
With this notation we have
\begin{equation}
\label{eq2.03} Y_t=\mathbb{E}(\Lambda_t|\mathcal{A}_t)
=\mathbb{E}(\Lambda_t|\GG_t\vee\FF_T^{\beta})
=\BE(\Lambda_t|\GG_t\vee\FF_{t,T}^{\beta}\vee\FF_t^{\beta})
=\BE(\Lambda_t|\FF_t\vee\FF_t^{\beta}).
\end{equation}
Since $f,g$ are measurable processes that are
$(\GG_T\vee\FF_{t,T}^\beta)$-adapted, they have modifications
which are $(\GG_T\vee\FF_{t,T}^\beta)$-progressively measurable.
Therefore the integrals on the right-hand side of (\ref{eq00.00})
are $(\GG_T\vee\FF_{t,T}^\beta)$-adapted. Hence
\[
\sigma(\Lambda_t)
\subset\GG_T\vee(\GG_T\vee\FF_{t,T}^B)=\GG_T\vee\FF_{t,T}^{\beta}.
\]
It follows in particular  that $\sigma(\Lambda_t)\vee\FF_t$ is
independent of $\FF_t^{\beta}$. By this, \cite[Proposition 5.6]{K}
and (\ref{eq2.03}),
\[
Y_t=\BE(\Lambda_t|\FF_t\vee\FF_t^{\beta})=\BE(\Lambda_t|\FF_t).
\]
Thus $Y$ is $(\FF_t)$-adapted. That $M$ is $(\FF_t)$-adapted now
follows from the equality
\[
M_t=Y_t-Y_0+\int_0^tf(r)\,dr+\int_0^tg(r)\,d^\dagger \beta_r.
\]
Step 2. Assume that $f$ is Lipschitz continuous in $y$ with
Lipschitz constant $L$. Let $(Y^0,M^0)=(0,0)$ and let
$(Y^{n+1},M^{n+1})$ be a solution of the equation
\begin{align}
\label{eq1.2}
Y^{n+1}_t&=\xi+\int_t^Tf(r,Y_r^n)\,dr+\int_t^Tg(r,Y_r^n)\,d^\dagger
\beta_r \\
&\quad+ \int_t^T\,dM_r^{n+1},\quad t\in[0,T],\quad\BP\mbox{-a.s.}
\nonumber
\end{align}
The sequence $\{(Y^n,M^n)\}$ is well defined, i.e. if
$(Y^n,M^n)\in\mathcal{S}^2\otimes \MM^2$ then $(Y^{n+1},M^{n+1})$
exists and belongs to $\mathcal{S}^2\otimes\MM^2$. To see this,
assume that $(Y^n,M^n)\in\mathcal{S}^2\otimes\MM^2$. Then
\[
\BE\int_0^T|f(t,Y_t^n)|^2\,dr\le 2 L^2\BE\int_0^T|Y_t^n|^2\,dt
+2\BE\int_0^T|f(t,0)|^2\,dt<\infty,
\]
\[
\BE\int_0^T\|g(t,Y_t^n)\|^2\,dt \le2 L^2\BE\int_0^T|Y_t^n|^2\,dt
+2\BE\int_0^T\|g(t,0)\|^2\,dt<\infty.
\]
Therefore by Step 1 there exists a solution $(Y^{n+1},M^{n+1})\in
\mathcal{S}^2\otimes\MM^2$ of (\ref{eq1.2}). By the It\^o-Meyer
formula and the assumption on $f$,
\begin{align}
\label{eq1.3}  &|Y^{n+1}_t-Y_t^n|^2+\int_t^T\,d[M^{n+1}-M^n]_r\\
&\qquad=
2\int_t^T(Y^{n+1}_r-Y^n_r)(f(r,Y_r^n)-f(r,Y_r^{n-1}))\,dr\nonumber\\
&\qquad\quad-2\int_t^T(Y^{n+1}_{r-}-Y^n_{r-})\,d(M^{n+1}-M^n)_r\nonumber\\
&\qquad\quad+2\int_t^T(Y^{n+1}_r-Y^n_r)(g(r,Y_r^n)-g(r,Y_r^{n-1}))\,d^\dagger
\beta_r \nonumber\\
&\qquad\quad+\int_t^T\|g(r,Y_r^n)-g(r,Y_r^{n-1})\|^2\,dr\nonumber\\
&\qquad\le \int_t^T|Y^{n+1}_r-Y_r^n|^2\,dr
+L^2\int_t^T|Y^{n}_t-Y_r^{n-1}|^2\,dr\nonumber\\
&\qquad\quad+l^2\int_t^T|Y^{n}_r-Y_r^{n-1}|^2\,dr
-2\int_t^T(Y^{n+1}_{r-}-Y_{r-}^n)\,d(M^{n+1}-M^n)_r\nonumber\\
&\qquad\quad-2\int_t^T(Y^{n+1}_r-Y_r^n)
(g(r,Y_r^n)-g(r,Y_r^{n-1}))\,d^\dagger \beta_r. \nonumber
\end{align}
Taking the expectation of  both sides of the above inequality and
using Gronwall's lemma we get
\[
\BE|Y^{n+1}_t-Y_t^n|^2+\BE\int_t^T\,d[M^{n+1}-M^n]_r \le
C\BE\int_t^T|Y^{n}_r-Y_t^{n-1}|^2\,dr,\quad t\in[0,T]
\]
for some $C$ depending only on $T,L,l$. Hence
\[
\sup_{t\le r\le T}\BE|Y^{n+1}_r-Y_r^n|^2
+\BE\int_t^T\,d[M^{n+1}-M^n]_r \le C(T-t)\sup_{t\le r\le
T}\BE|Y^n_r-Y^{n-1}_r|^2.
\]
Write
\[
\|(Y,M)\|^2_t=\sup_{t\le r\le T}\BE|Y_r|^2+\BE\int_t^Td[M]_r.
\]
With this notation we have
\begin{align*}
&\|(Y^{n+1},M^{n+1})-(Y^n,M^n)\|_{t}\\
&\qquad\le C(T-t)^{1/2}\|(Y^n,M^n)-(Y^{n-1}-M^{n-1})\|_{t}, \quad
n\ge 1,\quad t\in[0,T]
\end{align*}
From this one can deduce that
\[
\|(Y^{n+k},M^{n+k})-(Y^n,M^n)\|_{t_1}\le2^{-(n+1)}\|(Y^1,M^1)\|_{t_1}
\]
for $t_1\in[0,T]$ such that $C(T-t_1)=2^{-1}$.  Therefore dividing
the interval $[0,T]$ into small intervals and using a standard
argument one can show that
\[
\lim_{n,m\rightarrow\infty}\|(Y^m,M^m)-(Y^n-M^n)\|_0=0.
\]
Using the Burkholder-Davis-Gundy inequality we conclude from the
above  convergence and (\ref{eq1.3}) with $(Y^{n+1},M^{n+1})$
replaced by $(Y^m,M^m)$ that
\[
\lim_{n,m\rightarrow\infty}\BE\Big(\sup_{0\le t\le T}
|Y^n_t-Y^m_t|^2+\int_0^T\,d[M^n-M^m]_t\Big)=0.
\]
Let $(Y,M)$ be the limit of $\{(Y^n,M^n)\}$ in $\SM^2\times\MM^2$.
Letting $n\rightarrow\infty$ in (\ref{eq1.2}) shows that $(Y,M)$
is a solution of BDSDE$(\xi,f,g)$.
\smallskip\\
Step 3. Now we assume that $f$ satisfies the assumptions of the
theorem and moreover there exists $\lambda\in\BR$ such that
$f(t,y)\ge\lambda$ for a.e. $t\in[0,T]$ and every $y\in\BR$. Put
\[
f_n(t,y)=\inf_{x\in\mathbb{Q}}\{n|y-x|
+f(t,x)-Lx\}+Ly,\quad t\in[0,T],\quad y\in\BR.
\]
It is an elementary check that $f_n$ has the following properties:
for a.e. $t\in[0,T]$,
\begin{enumerate}
\item[(a)] $|f_n(t,y)-f_n(t,y')|\le(L+n)|y-y'|$ for all  $y,y'\in\BR$,
\item[(b)] $\lambda\le f_n(t,y)\le f(t,y)$ for every $y\in\BR$,
\item[(c)] $f_n(t,\cdot)\nearrow f(t,\cdot)$ uniformly on compact
subsets of $\BR$,
\item[(d)] $(f_n(t,y)-f_n(t,y'))(y-y')\le L|y-y'|^2$ for all
$y,y'\in\BR$.
\end{enumerate}
By Step 2, for each $n\ge 1$ there is a solution
$(Y^n,M^n)\in\mathcal{S}^2\otimes\MM^2$ of BDSDE$(\xi,f_n,g)$. By
Proposition \ref{stw1.1},
\begin{equation}
\label{eq1.5}
Y^n_t\le Y^{n+1}_t,\quad t\in[0,T], \quad\BP\mbox{-a.s.},\quad n\ge1.
\end{equation}
Put $Y_t=\sup_{n\ge 1}Y_t^n$. By Proposition \ref{stw1.2} and (b),
\begin{align}
\label{eq1.4} \BE\sup_{t\le T}|Y_t^n|^2+\BE[M^n]_T
&\le
C\BE\Big(|\xi|^2+\int_0^T(|f_n(t,0)|^2 +\|g(t,0)\|^2)\,dt\Big)
\\&\le C\BE\Big(|\xi|^2+\lambda^2+\int_0^T(|f(t,0)|^2
+\|g(t,0)\|^2)\,dt \Big). \nonumber
\end{align}
For every $\varepsilon,\eta>0$ we have
\begin{align*}
&\BP\Big(\int_0^T|f_n(t,Y^n_t)-f(t,Y_t)|\,dt>\varepsilon\Big)\\
&\quad\le
\BP\Big(\int_0^T|f_n(t,Y^n_t)-f(t,Y_t)|\,dt>\varepsilon,\,\,\sup_{n\ge1}
\sup_{t\le T}|Y^n_t|\le\eta\Big)+\BP(\sup_{n\ge 1}\sup_{t\le T}
|Y^n_t|\ge\eta)\\
&\quad\le\BP
\Big(\int_0^T|f_n(t,Y^n_t)-f(t,Y_t)|\,dt>\varepsilon,\,\,\sup_{n\ge1}
\sup_{t\le T}|Y^n_t|\le\eta\Big)\\
&\qquad+\eta^{-2}(\BE\sup_{t\le T} |Y^1_t|^2+\BE\sup_{t\le T}
|Y_t|^2),
\end{align*}
the last inequality being a consequence of (\ref{eq1.5}) and
Chebyshev's inequality. By (b), (c) and (A2) the first term on the
right-hand side of the above inequality tends to zero as
$n\rightarrow\infty$. The second one tends to zero as
$\eta\rightarrow\infty$ thanks to (\ref{eq1.4}). Hence
\begin{equation}
\label{eq1.6} \sup_{0\le t\le T}
|\int_0^t(f_n(r,Y^n_r)-f(r,Y_r))\,dr|\rightarrow 0
\end{equation}
in probability $\BP$ as $n\rightarrow\infty$.  By (\ref{eq1.5}),
(\ref{eq1.4}) and (A3) we also have
\begin{align}
\label{eq1.7} \BE\sup_{0\le t\le T}|
\int_0^t(g(r,Y^n_r)-g(r,Y))\,d^\dagger\beta_r|^2 &\le
4\BE\sum^{\infty}_{k=1}
\int_0^T|g_k(t,Y^n_t)-g_k(t,Y_t)|^2\,dt\\
&\le 4l\BE\int_0^T|Y_t^n-Y_t|^2\,dt, \nonumber
\end{align}
which converges to zero as $n\rightarrow\infty$. By Proposition
\ref{stw1.2},
\[
\sup_{n\ge 1}\BE\sup_{0\le t\le T}|\int_0^t
f_n(r,Y^n_r)\,dr|^2<\infty.
\]
Therefore letting
$n\rightarrow\infty$ in the equation
\[
Y^n_t=\mathbb{E}(\xi+\int_t^Tf_n(r,Y^n_r)\,dr
+\int_t^Tg(r,Y^n_r)\,d^\dagger\beta_r|\FF_t),\quad t\in[0,T]
\]
shows that $(Y,M)$, where
\[
M_t=\BE\Big(\xi+\int_0^Tf(r,Y_r)\,dr +\int_0^Tg(r,Y_r)\,d^\dagger
\beta_r|\FF_t\Big)-Y_0,
\]
is a solution of BDSDE$(\xi,f,g)$.
\smallskip\\
Step 4. We now show how to dispense with the assumption that $f$
is bounded from below. Let $f_n=f\vee(-n)$. By  Step 3, for each
$n\ge 1$ there exists a solution $(Y^n,M^n)$ of
BDSDE$(\xi,f_n,g)$. By Proposition \ref{stw1.1}, $Y_t^n\ge
Y_t^{n+1}$, $t\in[0,T]$, $\BP$-a.s. for $n\ge1$, whereas by
Proposition \ref{stw1.2},
\[
\BE\sup_{t\le T}|Y_t^n|^2+\BE\int_0^T\,d[M^n]_t\le C\BE
\Big(|\xi|^2+\int_0^T(|f(t,0)|^2+\|g(t,0)\|^2)\,dt\Big).
\]
Using the above properties of the processes $Y^n$ one can show in
much the same way as in Step 3 that (\ref{eq1.6}), (\ref{eq1.7})
hold true and then that there exists a solution $(Y,M)$ of
BDSDE$(\xi,f,g)$.
\end{proof}

\section{SPDEs and Markov-type BDSDEs}
\label{sec3}

In this section we first consider Markov-type BDSDEs. Roughly
speaking, these are  BDSDEs of the form (\ref{eq3.1}) with
filtration $(\GG_t)$ generated by some Markov process
$\BBM=(\BBX,P_z)$ and with final condition $\xi$ and coefficients
$f,g$ depending on $\omega'$ only through $\BBX(\omega')$. In our
paper $\BBM$ is a Markov process associated with a time-dependent
Dirichlet form. Using results of Section \ref{sec3} we prove the
existence and uniqueness of solutions of BDSDEs associated with
$\BBM$ for $\xi,f,g$ satisfying some ``markovian'' analogue of
conditions (A1)--(A4). Then we use this result to prove the
existence and uniqueness of solutions of SPDE of the form
(\ref{eq0.0}) with operator $\frac{\partial}{\partial t}+A_t$
associated with the underlying Dirichlet form.

\subsection{Dirichlet forms and Markov processes}
\label{sec3.1}

In what follows,  $E$ is a locally compact separable metric space
and $m$ is an everywhere dense Radon measure on $E$. By $\Delta$
we denote the one-point compactification of $E$. If $E$ is already
compact then we adjoin $\Delta$ to $E$ as an isolated point. We
set $E^1=\BR\times E$, $E_T=[0,T]\times E$, $E_{0,T}=(0,T]\times
E$, $m^1=l^1\otimes m$, $m_T=l^1|_{[0,T]}\otimes m$, where $l^1$
is the one-dimensional Lebesgue measure. We adopt the convention
that every function $\varphi$ on $E$ is extended to $E^1$ by
putting $\varphi(t,x)=\varphi(x)$, $(t,x)\in E^1$, and every
function $f$ on $E^1$ is extended to $E^1\cup\{\Delta\}$ by
putting $f(\Delta)=0$. Similarly, we extend functions defined on
$E_{0,T}$ or on $E_T$ to functions on $E^1\cup\{\Delta\}$ by
putting $f(z)=0$ outside $E_{0,T}$ or $E_T$, respectively.

We assume that we are given a family $\{B^{(t)},t\in\BR\}$ of
regular Dirichlet forms on $H=L^2(E;m)$ with sector constant
independent of $t$ and common domain $V\subset H$ (see, e.g.,
\cite{MR,Stannat} for the definition). We also assume that
\begin{enumerate}
\item[(a)] $\BR\ni t\mapsto B^{(t)}(\varphi,\psi)$ is measurable for
every $\varphi,\psi\in V$,

\item[(b)] there exist $c_1,c_2>0$ such that
$c_1B^{(0)}(\varphi,\varphi)\le B^{(t)}(\varphi,\varphi)\le
c_2B^{(0)}(\varphi,\varphi)$ for every $t\in\BR$ and $\varphi\in
V$.
\end{enumerate}
By assumption, $(B^{(0)},V)$ is closed, i.e. $V$ is a real Hilbert
space with respect to $\tilde B^{(0)}_1(\cdot,\cdot)$, where
$\tilde B^{(0)}(\varphi,\psi)
=\frac12(B^{(0)}(\varphi,\psi)+B^{(0)}(\psi,\varphi))$ and
$\tilde\BB^{(0)}_1(\varphi,\psi)
=\tilde\BB^{(0)}(\varphi,\psi)+(\varphi,\psi)_H$ for
$\varphi,\psi\in V$. By  $V'$ we denote the dual space of $V$ and
we set
\[
\VV=L^2(\BR;V),\quad \VV'=L^2(\BR;V'),\quad
\WW=\{u\in\VV:\frac{\partial u}{\partial t}\in \VV'\}.
\]

We will consider two time dependent Dirichlet forms $(\EE,D(\EE))$
and $(\EE^{0,T},D(\EE^{0,T}))$ associated with the families
$\{B^{(t)},t\in\BR\}$ and $\{B^{(t)},t\in[0,T]\}$, respectively.
The first one we define by putting
$D(\EE)=(\WW\otimes\VV)\cup(\VV\otimes\WW)$ and
\[
\mathcal{E}(u,v)=\left\{
\begin{array}{l}
\langle\frac{\partial u}{\partial t},v\rangle+
\int_\BR B^{(t)}(u(t),v(t))\,dt,\quad u\in\WW,\,v\in\VV,\smallskip \\
-\langle u,\frac{\partial v}{\partial t}\rangle+ \int_\BR
B^{(t)}(u(t),v(t))\,dt,\quad u\in\VV,\,v\in\WW,
\end{array}
\right.
\]
where  $\langle\cdot,\cdot\rangle$ denotes the duality pairing
between $\VV$ and $\VV'$. It is known (see \cite[Example
I.4.9]{Stannat}) that $(\EE,D(\EE))$ is a generalized  Dirichlet
form on $L^2(E^1;m^1)$.

By \cite[Theorem 4.2]{O1} (see also \cite[Theorem 6.3.1]{O3})
there exists a Hunt process
$\mathbf{M}=(\Omega',(\mathbf{X}_t)_{t\ge 0},(P_z)_{z\in
E^1\cup\Delta}$, $(\FF_t^\BBX)_{t\ge 0})$  with state space $E^1$,
life time $\zeta$ and cemetery state $\Delta$ properly associated
with $(\EE,D(\EE))$ in the resolvent sense. Moreover,
\begin{equation} \label{eq3.17}
\BBX_t=(\iota(t),X_{\iota(t)}),\quad t\ge0,
\end{equation}
where $\iota$ is the uniform motion to the right, i.e.
$\iota(t)=\iota(0)+t$, $\iota(0)=s$, $P_{z}$-a.s. for $z=(s,x)$.
One can also check that  $(X,(P_{s,x})_{x\in E})$ is a time
inhomogeneous Markov process associated with the family
$\{(B^{(t)},V),t\ge0\}$, i.e. for every $s\in\BR$,
$(X_{s+\cdot}\,,(P_{s,x})_{x\in E})$ is a Hunt  process
associated with the form $(B^{(s)},V)$.

Let
\[
\WW_T=\{u\in L^2(0,T;V):\frac{\partial u}{\partial t}\in
L^2(0,T;V'),\, u(T)=0\},
\]
\[
\WW_0=\{u\in L^2(0,T;V):\frac{\partial u}{\partial t}\in
L^2(0,T;V'),\, u(0)=0\}.
\]
We define the second form  by setting $D(\EE^{0,T})=(\WW_T\otimes
L^2(0,T;V))\cup (\WW_0\otimes L^2(0,T;V))$ and
\[
\mathcal{E}^{0,T}(u,v)=\left\{
\begin{array}{l}\langle\frac{\partial u}{\partial t},v\rangle+
\BB^{0,T}(u,v),\quad u\in\WW_0,\,v\in L^2(0,T;V),\smallskip \\
-\langle u,\frac{\partial v}{\partial t}\rangle+
\BB^{0,T}(u,v),\quad v\in\WW_T,\,u\in L^2(0,T;V),
\end{array}
\right.
\]
where now $\langle\cdot,\cdot\rangle$ denotes the duality pairing
between $L^2(0,T;V')$ and $L^2(0,T;V)$, and
\begin{equation}
\label{eq3.16} \BB^{0,T}(u,v)=\int_0^T
B^{(t)}(u(t),v(t))\,dt,\quad u,v\in L^2(0,T;V).
\end{equation}
Note that by \cite[Example I.4.9]{Stannat},
$(\EE^{0,T},D(\EE^{0,T}))$ is a generalized Dirichlet form on
$L^2(E_{0,T};m_T)$.

We set
\[
R_\alpha f(z)=E_z\int_0^\infty e^{-\alpha t}f(\BBX_t)\,dt, \quad
z\in E^1,\quad f\in \BB(E^1),\quad\alpha\ge 0,
\]
\[
R^{0,T}_\alpha f(z)=E_z\int_0^{T-\iota(0)}e^{-\alpha
t}f(\BBX_t)\,dt, \quad z\in E_{0,T},\quad f\in
\BB(E_{0,T}),\quad\alpha\ge 0
\]
whenever the integrals exist. Let $\hat{\mathbf{M}}=(\hat\BBX,\hat
P_z)$ be the dual process of $\mathbf{M}$ (see \cite[Section
3.3]{O3}). We set (whenever the integrals exist)
\[
\hat{R}_\alpha f(z)=\hat{E}_z\int_0^\infty e^{-\alpha t} f(\hat
\BBX_t)\,dt, \quad z\in E^1,\quad f\in\BB(E^1),\quad\alpha\ge 0,
\]
\[
\hat{R}^{0,T}_\alpha f(z) =\hat{E}_z\int_0^{\iota(0)} e^{-\alpha
t} f(\hat \BBX_t)\,dt, \quad z\in E_{0,T},\quad f\in
\BB(E_{0,T}),\quad\alpha\ge 0.
\]
To shorten notation, we write $R=R_0$, $\hat{R}=\hat{R}_0$,
$R^{0,T}=R_0^{0,T}$, $\hat{R}^{0,T}=\hat{R}_0^{0,T}$. It is well
known that $G_\alpha=R_\alpha$, $\hat{G}_\alpha=\hat{R}_\alpha$ on
$L^2(E^1;m^1)$, $\alpha>0$, where $\{G_\alpha,\alpha>0\}$ is the
resolvent and  $\{\hat G_\alpha,\alpha>0\}$ is the dual resolvent
associated with the form  $(\EE,D(\EE))$. Similarly,
$G^{0,T}_\alpha =R_\alpha^{0,T}$,
$\hat{G}^{0,T}_\alpha=\hat{R}_\alpha^{0,T}$ on $L^2(E_{0,T};m_T)$,
$\alpha\ge 0$, where $\{G^{0,T}_\alpha,\alpha\ge0\}$ is the
resolvent and $\{\hat{G}^{0,T}_\alpha,\alpha\ge0\}$ is the dual
resolvent associated with $(\EE^{0,T},D(\EE^{0,T}))$.

By  $A_t$ we  denote the operator associated with the form
$(B^{(t)},V)$, i.e.
\begin{equation}
\label{eq3.9} (-A_tu,v)=B^{(t)}(u,v), \quad u\in D(A_t),v\in V,
\end{equation}
where $D(A_t)=\{u\in V:v\mapsto B^{(t)}(u,v)$  is continuous with
respect to $(\cdot,\cdot)^{1/2}$ on $V$\} (see \cite[Proposition
I.2.16]{MR}).

Let cap be the capacity considered in \cite{O2} (see also
\cite[Section 6.2]{O3}). We say that a set $B$ is
$\EE$-exceptional if cap$(B)=0$. We say that some property is
satisfied quasi-everywhere (q.e. for short) if the set of those
$z\in E^1$ for  which it does not hold is $\EE$-exceptional. Note
that a nearly Borel set $B$ is $\EE$-exceptional if and only if it
is $\BBM$-exceptional, i.e. $P_{m_1}(\sigma_B<\infty)=0$, where
$\sigma_B=\inf\{t>0:\BBX_t\in B\}$ (see, e.g., \cite[p. 298]{O2}).

We say that $f:E_{0,T}\rightarrow\BR$ is quasi-integrable ($f \in
qL^1$ in notation) if $f$ is Borel measurable and
$P_z(\int^{T-\iota(0)}_0|f(\BBX_t)|\,dt<\infty)=1$ for q.e. $z\in
E_{0,T}$. By $q\mathbb{L}^1$ we denote the set of measurable
functions $f:\Omega\times E_{0,T}\rightarrow \BR$ such that $P(\{
\omega\in\Omega; \, f(\omega,\cdot)\in qL^1\})=1$.

\subsection{Existence and uniqueness of solutions}

Let $A_t$ be the operator defined by (\ref{eq3.9}). Suppose we are
given measurable functions $\varphi: E\rightarrow\BR$,
$f:\Omega\times E_T\times\BR\rightarrow\BR$, $g_k:\Omega\times
E_T\times\BR\rightarrow\BR$, $k\in\BN$, and let
$g=(g_1,g_2,\dots)$. We are going to show that there exists a
unique solution of the SPDE
\begin{equation}
\label{eq4.1}
du(t)=-(A_tu+f(t,x,u))\,dt-g(t,x,u)\,d^{\dagger}\beta_t, \quad
u(T)=\varphi,
\end{equation}
under hypotheses (H1)--(H5) given below:
\begin{enumerate}
\item[(H1)] $\|\varphi\|_{L^2(E;m)}+E\|f(\cdot,0)\|^2_{L^2(E_{0,T};m_T)}
+E\sum_{k=1}^\infty\|g_k(\cdot,0)\|^2_{L^2(E_{0,T};m_T)}<\infty$.
\item[(H2)] For every $y\in\BR$ the mapping $E_{0,T}\ni z\mapsto f(z,y)$
belongs to $q\BL^1$.
\item[(H3)]There exist $l,L>0$ and measurable functions
$L_k:E_{0,T}\rightarrow[0,\infty)$ such that $\sup_{z\in
E_{0,T}}\sum_{k=1}^{\infty}L^2_k(z)\le l$ and for every $z\in
E_{0,T}$ and $y,y'\in\BR$,
\begin{enumerate}
\item[(a)]
$(f(z,y)-f(z,y'))(y-y')\le L|y-y'|^2$,

\item[(b)]
$|g_k(z,y)-g_k(z,y')| \le L_k(z)|y-y'|$.
\end{enumerate}
\item[(H4)]For every $\omega\in\Omega$ and $z\in E_{0,T}$
the mapping $\BR\ni y\mapsto f(z,y)$ is continuous.

\item[(H5)] For every $y\in\BR$ the mappings
$f(\cdot,y),g_k(\cdot,y):\Omega\times E_T\rightarrow\BR$,
$k\in\BN$, are $(\FF^{\beta}_{t,T})$-progressively measurable,
i.e. for every $t\in[0,T]$ the mappings $\Omega\times[T-t,T]\times
E\ni(\omega,s,x) \mapsto f(\omega,s,x,y)$ and
$\Omega\times[T-t,T]\times E\ni(\omega,s,x) \mapsto
g_k(\omega,s,x,y)$ are $\FF^{\beta}_{t,T}\otimes\BB([T-t,T])
\otimes\BB(E)$-measurable.
\end{enumerate}

In what follows
\begin{equation}
\label{eq3.5} \beta_t^{\iota}=\beta_{t+\iota(0)},\quad
T_{\iota}=T-\iota(0),\quad
\FF_t=\FF_{t,T_\iota}^{\beta^\iota}\vee\FF_t^{\BBX}, \quad
t\in[0,T],
\end{equation}
where $\FF_{t,T_\iota}^{\beta^\iota}=\sigma(\beta^{\iota}_{r\wedge
T_{\iota}}-\beta^{\iota}_t,r\in[t,T])$ (see our convention
(\ref{eq2.02})).

We will say that a random function $u:\Omega\times
E_{0,T}\rightarrow\BR$ is adapted if $u(\BBX)\in M$, where  $M$ is
defined as in Section \ref{sec2} but with respect to $(\FF_t)$
defined by (\ref{eq3.5}).

\begin{remark}
\label{rem3.1} (i) Assume (H5). Then for every $y\in\BR$ there
exists a $P\otimes m_T$-version $\tilde f(\cdot,y)$ of
$f(\cdot,y)$ and a $P\otimes m_T$-version $\tilde g_k(\cdot,y)$ of
$g_k(\cdot,y)$ such that the processes $\tilde f(\BBX,y), \tilde
g_k(\BBX,y)$ are of class $M$ under $\BP_z$ for q.e. $z\in
E_{0,T}$. To see this, we let $h$ stand for $f(\cdot,y)$ or for
$g_k(\cdot,y)$. It is known that any
$(\FF^{\beta}_{t,T})$-progressively measurable square-integrable
$h:\Omega\times E_T\rightarrow\BR$ may be approximated in the
space $L^2(\Omega\times E_T;P\otimes m_T)$ by linear combinations
of processes of the form
\begin{equation}
\label{eq3.6} S_j(\omega,t,x)=\mathbf{1}_{\Lambda_j}(\omega)
\mathbf{1}_{(T-r_{j+1},T-r_j]}(t) b_j(x),
\end{equation}
where $0\le r_j\le r_{j+1}\le T$, $\Lambda_j\in
\FF^\beta_{T-r_j,T}$  and $\{b_j,\, j\ge 0\}$ is some orthonormal
basis of $H$. It is an elementary check that the process
$\{S_j(\mathbf{X}_t),t\ge 0\}$ is $\{\FF^{\beta^\iota}_{t,
T_\iota}\}$-adapted. Therefore linear combinations of processes of
the form  (\ref{eq3.6}) are of class $M$.  Hence there is a
sequence $\{h_n\}$ such that $h_n\rightarrow h$ in
$L^2(\Omega\times E_{0,T};P\otimes m_T)$ and the processes
$h_n(\BBX)$ are of class $M$. Let $\tilde
h=\limsup_{n\rightarrow\infty}h_n$. Then $\tilde h(\BBX)$ is of
class $M$ and  $\tilde h$ is a $P\otimes m_T$-version of $h$. Let
$\hat S_{00}$ denote the set of all  finite zero order integral
measures $\nu$ on $E_{0,T}$ such that $\nu(E_{0,T})<\infty$ and
$\|\hat{R}^{0,T}\nu\|_{\infty}<\infty$. For every $\nu\in\hat
S_{00}$ we have
\[
\mathbb{E}_{\nu}\int_0^{T_\iota} |(\tilde
h(\mathbf{X}_t)-h(\BBX_t)|^2\,dt \le E\|\tilde
h-h\|^2_{L^2(E_{0,T};m_T)}\|\hat R^{0,T}\nu\|_{\infty}=0.
\]
From this and \cite[Theorem 2.5]{Trutnau} (or remark at the end of
Section 4 in \cite{O2}) it follows that $\BP_z(\tilde
h(\BBX_t)=h(\BBX_t)$ for a.e. $t\in[0,T])=1$ for q.e. $z\in
E_{0,T}$, which shows the desired result in case $h$ is
square-integrable. The general case is handled by approximating
$h$ pointwise by square-integrable functions.
\smallskip\\
(ii) From the fact that Carath\'eodory functions are jointly
measurable it follows that if $f,g_k$ satisfy (H3b), (H4), (H5)
and $u:\Omega\times E_{0,T}\rightarrow\BR$ is adapted then the
processes $t\mapsto\tilde f(\BBX_t,u(\BBX_t))$, $t\mapsto \tilde
g_k(\BBX_t,u(\BBX_t))$ are of class $M$.
\end{remark}

We first give the definition of a solution of (\ref{eq4.1}) and
related Markov-type BDSDE with coefficients $f,g$. In what follows
we always assume that the coefficients $f,g_k,k\ge1$, satisfy
(H5), and we always assume that we are taking their versions
having the properties listed in Remark \ref{rem3.1}.

\begin{definition}
We say that an adapted function $u:\Omega\times
E_{0,T}\rightarrow\BR$ is a solution of SPDE (\ref{eq4.1}) if
\begin{enumerate}
\item[(a)] $\BP_z(\int_0^{T_{\iota}}|f(\BBX_t,u(\BBX_t))|\,dt<\infty)=1$,
$\BE_z\sup_{0\le t\le T_\iota}
|\int_0^tf(\BBX_r,u(\BBX_r))\,dr|^2<\infty$ and
$\BE_z\int_0^{T_\iota}\|g(\BBX_t,u(\BBX_t))\|^2\,dt<\infty$ for
q.e. $z\in E_{0,T}$, where $\BP_z=P\otimes P_z$ and $\BE_z$
denotes the expectation with respect to $\BP_z$,

\item [(b)]for $P$-a.e. $\omega\in\Omega$,
$u(\omega;\cdot):E_{0,T}\rightarrow\BR$ is quasi-continuous and
for q.e. $z\in E_{0,T}$\,,
\[
u(z)=E_z\Big(\varphi(\BBX_{T_\iota})
+\int_0^{T_\iota}f(\BBX_t,u(\BBX_t))\,dt
+\int_0^{T_\iota}g(\BBX_t,u(\BBX_t))\,d^\dagger\beta_t^{\iota}\Big).
\]
\end{enumerate}
\end{definition}

\begin{definition}
A pair $(Y^z,M^z)\in\SM^2\times\MM^2$  is a solution of the BDSDE
\begin{align}
\label{eq3.4}
Y^z_t&=\varphi(\BBX_{T_{\iota}})+\int^{T_{\iota}}_t
f(\BBX_r,Y^z_r)\,dr\\
&\quad+\int^{T_{\iota}}_tg(\BBX_r,Y^z_r)\,d^{\dagger}\beta^{\iota}_r
-\int^{T_{\iota}}_tdM^z_r,\quad t\in[0,T_{\iota}] \nonumber
\end{align}
(BDSDE${}_z(\varphi,f,g)$ for short) if
\begin{enumerate}
\item[(a)]$\int_0^{T_{\iota}}|f(\BBX_t,Y^z_t)|\,dt<\infty$,
$\int_0^{T_{\iota}}\|g(\BBX_t,Y^z_t)\|^2\,dt<\infty$, $\BP_z$-a.s.
\item[(b)] Equation (\ref{eq3.4}) is satisfied $\BP_z$-a.s.
\end{enumerate}
\end{definition}
\medskip

We first prove that under (H1)--(H5) for q.e. $z\in E_{0,T}$ there
exists a unique solution of (\ref{eq3.4}), and moreover, one can
find a version of the solution which is independent of $z$. In
fact, the solution exists for every $z\in E_{0,T}\setminus N$,
where
\begin{equation}
\label{eq3.07} N=N_1\cup N_2,
\end{equation}
with
\[
N_1=\{z\in E_{0,T}:\BE_z \Big(|\varphi(\BBX_{T_\iota})|^2
+\int_0^{T_\iota}(|f(\BBX_t,0)|^2 +\|g(\BBX_t,0)\|^2)\,dt\Big)
=\infty\},
\]
\[
N_2=\{z\in E_{0,T}:\exists_{y\in\BR}\,\,\BP_z
\Big(\int_0^{T_\iota}|f(\BBX_r,y)|\,dr=\infty\Big)>0\}.
\]
\begin{lemma}
\label{stw2.1} Let  \mbox{\rm(H1)--(H5)} hold and let $N$ be
defined by \mbox{\rm(\ref{eq3.07})}. Then \mbox{\rm cap$(N)=0$}
and for every $z\in E_{0,T}\setminus N$ the data
\[
(\xi,\bar f(t,y),\bar g(t,y))= (\varphi(\BBX_{T_\iota}),
f(\BBX_t,y),g(\BBX_t,y)),\quad t\in[0,T],y\in\BR
\]
satisfy  assumptions $\mbox{\rm{(A1)--(A4)}}$ under the measure
$\BP_z$.
\end{lemma}
\begin{proof}
Let $N_0=\{z\in E_{0,T};\BE_z\int_0^{T_{\iota}}
|f(\BBX_t,0)|^2\,dt=\infty\}$. We may assume that $N_0$ is
compact. Let $\sigma_{N_0}$ be the first hitting time of $N_0$.
Then
\begin{align*}
\BP_z(\sigma_{N_0}<\zeta_\iota)&=\BP_z(\exists_{t<\zeta_\iota}\BBX_t\in
N_0) \le \BP_z\Big(\BE_{\BBX_{\sigma_{N_0}}}
\int_0^{T_\iota}|f(\BBX_t,0)|^2\,dt=\infty\Big)\\
&=\BP_z\Big(\BE_z\big(\int_{\sigma_{N_0}}^{T_\iota}|f(\BBX_t,0)|^2\,dt
|\FF_{\sigma_{N_0}}^\BBX\big)=\infty\Big)\\
&\le\BP_z\Big(\BE_z\big(\int_0^{T_\iota}|f(\BBX_t,0)|^2\,dt
|\FF_{\sigma_{N_0}}^\BBX\big)=\infty\Big)=0
\end{align*}
for $m_T$-a.e. $z\in E_{0,T}$, because for any strictly positive
Borel function $\psi$ on $E_{0,T}$ such that
$\|\hat{R}^{0,T}\psi\|_\infty<\infty$ we have
\begin{align*}
&\int_{E_{0,T}}\Big(\BE_z\int_0^{T_\iota}|f(\BBX_t,0)|^2\,dt\Big)
\psi(z)m_T\,(dz)
=E(R^{0,T}f^2(\cdot,0),\psi)_{L^2(E_{0,T};m_T)}\\
&\quad=E(f^2(\cdot,0),\hat{R}^{0,T}\psi)_{L^2(E_{0,T};m_T)} \le
E\|f(\cdot,0)\|^2_{L^2(E_{0,T};m_T)}\cdot\|R^{0,T}\psi\|_\infty<\infty.
\end{align*}
Hence cap$(N_0)=0$. In much the same way one can show that
cap$(N_1)=0$. Moreover, by the definition of the space $q\BL^1$,
cap$(N_2)=0$. Thus cap$(N)=0$ and (A1)--(A4) are satisfied under
the measure $\BP_z$ for $z$ outside the set $N=N_1\cup N_2$.
\end{proof}
\medskip

In what follows we use the notation
\[
\mathcal{A}_t=\FF^{\beta^\iota}_{T_\iota}\vee\FF^\BBX_t.
\]

\begin{theorem}
\label{stw2.2} Assume that $\varphi,f,g$ satisfy
\mbox{\rm(H1)--(H5)} and define $N$ by \mbox{\rm(\ref{eq3.07})}.
Then for every $z\in E_{0,T}\setminus N$ there exists a solution
$(Y^z,M^z)$ of $\mbox{\rm{BDSDE}}_z(\varphi,f,g)$. Moreover, there
exists a pair of c\`adl\`ag $(\FF_t)$-adapted processes $(Y,M)$
such that for every $z\in E_{0,T}\setminus N$,
\begin{equation}
\label{eq2.1} (Y_t^z,M_t^z)=(Y_t,M_t),\quad
t\in[0,T_\iota],\quad\BP_z\mbox{\rm{-a.s.}}
\end{equation}
\end{theorem}
\begin{proof}
The existence of a solution $(Y^z,M^z)$ follows from Theorem
\ref{tw1.1} and  Lemma \ref{stw2.1}. The proof of (\ref{eq2.1}) we
divide into three steps.
\smallskip\\
Step 1. Assume that $f,g$ do not depend on the last variable $y$.
Then
\begin{equation}
\label{eq3.8} M_t^z=\mathbb{E}_z\Big(\varphi(\BBX_{T_\iota})
+\int_0^{T_\iota}f(\BBX_r)\,dr
+\int_0^{T_\iota}g(\BBX_r)\,d^\dagger\beta_r^{\iota}
|\mathcal{A}_t\Big)-Y^z_0
\end{equation}
and
\[
Y_t^z=\mathbb{E}_z\Big(\varphi(\BBX_{T_\iota})
+\int_t^{T_\iota}f(\BBX_r)\,dr
+\int_t^{T_\iota}g(\BBX_r)\,d^\dagger \beta_r^{\iota}
|\mathcal{A}_t\Big).
\]
By \cite[Lemma A.3.5]{FOT} there exists a random variable $H_0$
such that $Y_0^z=H_0$, $\mathbb{P}_z$-a.s. for every $z\in
E_{0,T}\setminus N$, while by \cite[Lemma A.3.3, A.3.5]{FOT} there
exists an $(\mathcal{A}_t)$-adapted c\`adl\`ag process $M$ such
that $M_t=M_t^z$, $t\in[0,T_\iota]$ $\BP_z$-a.s. for $z\in
E_{0,T}\setminus N$.  Set
\[
Y_t=H_0-\int_0^tf(\mathbf{X}_r)\,dr
-\int_0^tg(\mathbf{X}_r)\,d^{\dagger}\beta_r+\int_0^t\,dM_r,\quad
t\in[0,T_\iota].
\]
Then $Y$ is an $(\FF_t)$-adapted c\`adl\`ag process  such that
$Y_t=Y_t^z$, $t\in[0,T_\iota]$, $\BP_z$-a.s. for every $z\in
E_{0,T}\setminus N$.
\smallskip\\
Step 2. We now consider general $f,g$ (possibly depending on $y$)
but we assume that $f$ is Lipschitz continuous with respect to $y$
uniformly in $t$. By Step 2 of the proof Theorem \ref{tw1.1},
\begin{equation}
\label{eq2.2} (Y^{z,n},M^{z,n})\rightarrow(Y^z,M^z)\mbox{ in
}S^2\otimes\MM^2,
\end{equation}
where $(Y^{z,0},M^{z,0})=(0,0)$ and
\begin{align*}
Y_t^{z,n+1}&=\varphi(\BBX_{T_\iota})
+\int_t^{T_\iota}f(\BBX_r,Y_r^{z,n})\,dr+
\int_t^{T_\iota}g(\BBX_r,Y_r^{z,n})\,d^{\dagger}\beta_r\\
&\quad-\int_t^{T_\iota}\,dM_r^{z,n+1},\quad t\in[0,T_\iota],\quad
\BP_z\mbox{-a.s.}
\end{align*}
for $z\in E_{0,T}\setminus N$. By Step 1, for every $n\ge 0$ there
exists a pair $(Y^n,M^n)$ of $(\FF_t)$-adapted
 c\`adl\`ag processes such that (\ref{eq2.1}) holds for
$(Y^{z},M^{z})$, $(Y,M)$ replaced by $(Y^{z,n},M^{z,n})$,
$(Y^n,M^n)$. Therefore applying \cite[Lemma A.3.3]{FOT} we show
the existence of a pair $(Y,M)$ satisfying (\ref{eq2.1}).
\smallskip\\
Step 3. The general case. From the proof of Theorem \ref{tw1.1} it
follows that (\ref{eq2.2}) holds for $(Y^{z,n},M^{z,n})$ being a
solution of BDSDE$_z(\varphi,f_n,g)$ with some  Lipschitz
continuous in $y$ function $f_n$. By Step 2, (\ref{eq2.1}) holds
for $(Y^{z,n},M^{z,n})$. Therefore applying  \cite[Lemma
A.3.3]{FOT} shows that  (\ref{eq2.1}) is satisfied.
\end{proof}
\medskip

Of course the pair $(Y,M)$ of Theorem \ref{stw2.2} is a solution
of $\mbox{\rm{BDSDE}}_z(\varphi,f,g)$ for q.e. $z\in E_{0,T}$. Our
next goal is to show that the function $u$ defined as
$u(z)=E_zY_0$ for $z\in E_{0,T}\setminus N$ is a solution of SPDE
(\ref{eq4.1}). We first prove this in Proposition \ref{stw3.41}
for linear equations and then in Theorem \ref{tw3.4} in the
general case. We begin with a simple but useful lemma.

\begin{lemma}
\label{lm3.333} Assume that $\Lambda\in\mathcal{A}_{T_\iota}$ and
let $N$ be a properly exceptional subset of $E_{0,T}$. If
$\BP_z(\Lambda)=1$ for $z\in E_{0,T}\setminus N$ then for every
$z\in E_{0,T}\setminus N$, $\BP_z(\theta_\tau^{-1}(\Lambda))=1$
for every stopping time $\tau$ with respect to $(\mathcal{A}_t)$
such that $0\le\tau\le T_\iota$.
\end{lemma}
\begin{proof}
By the strong Markov property and proper exceptionality of $N$, for every
$z\in E_{0,T}\setminus N$ we have
\[
\BP_z(\theta_\tau^{-1}(\Lambda^c))=\BE_z\BP_{\BBX_\tau}(\Lambda^c)=0,
\]
which proves the lemma.
\end{proof}

\begin{remark}
It is clear that if $\tau$ is an $(\mathcal{A}_t)$-stopping time
then $\tau(\omega;\cdot)$ is an $(\FF_t^{\mathbf{X}})$-stopping
time for every $\omega\in \Omega$.
\end{remark}

In the sequel we  extend the shift operator $\theta$ to
$\Omega\times \Omega'$ by putting
\[
\theta_t(\omega,\omega')=(\omega,\theta_t(\omega')).
\]
It is clear that $\theta_t^{-1}(\mathcal{A}_0)\subset\mathcal{A}_t$.

\begin{proposition}
\label{stw3.41} Assume that $\varphi,f,g$ satisfy $\mbox{\rm(H1),
(H2), (H5)}$ and $f,g$ do not depend on the last variable $y$.
Define $N$ by \mbox{\rm(\ref{eq3.07})} and set
\begin{equation}
\label{eq2.71} u(z)=E_z\Big(\varphi(\BBX_{T_\iota})
+\int_0^{T_\iota}f(\BBX_t)\,dt+\int_0^{T_\iota}g(\BBX_t)\,d^\dagger
\beta_t^\iota\Big)
\end{equation}
for $z\in E_{0,T}\setminus N$ and $u(z)=0$  otherwise. Then for
$P$-a.s. $\omega\in\Omega$ the function $u(\omega;\cdot)$ is
quasi-continuous and for every $z\in E_{0,T}\setminus N$,
\begin{equation}
\label{eq2.31} Y_t=u(\BBX_t),\quad
t\in[0,T_\iota],\quad\BP_z\mbox{-a.s.},
\end{equation}
where $Y$ is the process of Theorem \ref{stw2.2}.
\end{proposition}
\begin{proof}
By \cite[Theorem 4.1.1]{FOT} we may and will assume that $N$ is
properly exceptional. Let $(Y,M)$ be the pair of processes of
Theorem \ref{stw2.2}. We shall show  that for every $z\in
E_{0,T}\setminus N$,
\begin{equation}
\label{eq2.61} Y_0\circ\theta_\tau=Y_\tau,\quad
\mathbb{P}_z\mbox{-a.s.}
\end{equation}
for every $(\mathcal{A}_t)$-stopping time $\tau$ such that $0\le
\tau\le T_\iota$. To this end,  first observe  that
$\mathbb{P}_z\mbox{-a.s.}$ we have
\begin{equation}
\label{eq2.41}
\Big(\int_0^{T_\iota}f(\BBX_t)\,dt\Big)\circ\theta_\tau
=\int_\tau^{T_\iota}f(\BBX_t)\,dt,
\end{equation}
\begin{equation}
\label{eq2.42} \Big(\int_0^{T_\iota}g(\BBX_t)\,d^\dagger
\beta_t^\iota\Big)\circ\theta_\tau
=\int_\tau^{T_\iota}g(\BBX_t)\,d^\dagger\beta_t^\iota.
\end{equation}
Indeed, the first equation holds since $A_t=\int_0^tf(\BBX_r)\,dr$
is a continuous AF of $\BBM$ for fixed $\omega\in\Omega$. The
second one may be deduced from the identity
\[
\Big(\sum_{0\le t_i\le T_\iota}g(\BBX_{t_{i+1}})
(\beta^\iota_{t_{i+1}}-\beta^\iota_{t_i})\Big)\circ
\theta_\tau=\sum_{0\le t_i\le T_\iota-\tau}g(\BBX_{t_{i+1}+\tau})
(\beta^\iota_{t_{i+1}+\tau}-\beta^\iota_{t_\iota+\tau}),
\]
which holds $\mathbb{P}_z$-a.s. Here we used the fact that
$\iota(0)\circ\theta_\tau=\iota(0)+\tau$. It is also clear that
\begin{equation}
\label{eq2.142} \varphi(\mathbf{X}_{T_\iota})\circ\theta_\tau
=\varphi(\mathbf{X}_{T_\iota}).
\end{equation}
By \cite[Theorem 50.19]{Sharpe},
$(N_t=(M_{t-\tau}\circ\theta_\tau-M_0\circ\theta_\tau)
\mathbf{1}_{[\tau,\infty)}(t),\FF^{\mathbf{X}}_t, t\ge 0)$ is a
martingale for fixed $\omega\in\Omega$. Hence
\begin{equation}
\label{eq501} \mathbb{E}_z\Big((\int_0^{T_\iota}\,dM_t)
\circ\theta_\tau|\mathcal{A}_\tau\Big)=
E_z\Big((\int_0^{T_\iota}\,dM_t)
\circ\theta_\tau|\FF_\tau^{\mathbf{X}}\Big)
=E_z(N_{T_\iota+\tau}|\FF_\tau^{\mathbf{X}})=0.
\end{equation}
We know that for every $z\in E_{0,T}\setminus N$,
\[
Y_0=\varphi(\mathbf{X}_{T_\iota})
+\int_0^{T_\iota}f(\mathbf{X}_r)\,dr
-\int_0^{T_\iota}g(\mathbf{X}_r)\,d^{\dagger}\beta^{\iota}_r
-\int_0^{T_\iota} dM_r,\quad \mathbb{P}_z\mbox{-a.s.}
\]
By Lemma \ref{lm3.333} and (\ref{eq2.41})--(\ref{eq2.142}),  for
every $z\in E_{0,T}\setminus N$,
\[
Y_0\circ \theta_\tau=\varphi(\mathbf{X}_{T_\iota})
+\int_\tau^{T_\iota}f(\mathbf{X}_r)\,dr
-\int_\tau^{T_\iota}g(\mathbf{X}_r)\,d^{\dagger}\beta^{\iota}_r
-(\int_0^{T_\iota}dM_r)\circ\theta_\tau,\quad
\mathbb{P}_z\mbox{-a.s.}
\]
Since $Y_0$ is $\mathcal{A}_0$-measurable, $Y_0\circ\theta_\tau$
is $\mathcal{A}_\tau$-measurable. Therefore by (\ref{eq501}),
\[
Y_0\circ \theta_\tau
=\mathbb{E}_z\Big(\varphi(\mathbf{X}_{T_\iota})
+\int_\tau^{T_\iota}f(\mathbf{X}_r)\,dr
-\int_\tau^{T_\iota}g(\mathbf{X}_r)\,d^{\dagger}\beta^{\iota}_r
|\mathcal{A}_\tau\Big),\quad \mathbb{P}_z\mbox{-a.s.}
\]
Thus (\ref{eq2.61}) is satisfied for $z\in E_{0,T}\setminus N$.
Since $u(z)=E_zY_0$, using the strong Markov property and
(\ref{eq2.61}) we see that for every $z\in E_{0,T}\setminus N$,
\[
u(\BBX_\tau)=E_{\BBX_\tau}Y_0=E_z(Y_0\circ\theta_\tau|\FF_\tau^X)=
E_z(Y_\tau|\FF_\tau^X)=\mathbb{E}_z(Y_\tau|\mathcal{A}_\tau)
=Y_\tau,\quad \mathbb{P}_z\mbox{-a.s.}
\]
A standard approximation argument shows that the process
$u(\mathbf{X})$  is optional (with respect to $(\mathcal{A}_t)$).
Therefore the above equality implies (\ref{eq2.31})  by the
Section Theorem. To prove quasi-continuity of $u(\omega;\cdot)$
for $P$-a.s. $\omega\in\Omega$,  fix $\omega\in\Omega$ such that
(\ref{eq2.31}) holds $P_{m_T}$-a.s. Let $\tau$  be a predictable
$(\FF^{\mathbf{X}}_t)$-stopping time such that $0\le\tau\le
T_\iota$. Since $\mathbf{X}$ is a Hunt process,
$X_{\tau-}=X_\tau$, $P_z$-a.s. for every $z\in E_{0,T}$,  and
moreover, the filtration $(\FF^{\mathbf{X}}_t)$ is quasi-left
continuous, so  $M_{\tau-}=M_\tau$, $P_z$-a.s. for every $z\in
E_{0,T}$ (see \cite[Theorems A.3.2, A.3.6]{FOT}). By this and
(\ref{eq2.31}),
\[
u(\mathbf{X}_{\tau-})=u(\mathbf{X}_\tau)
=u(\mathbf{X})_{\tau-}\,,\quad P_{m_T}\mbox{-a.s.}
\]
Of course the process $\{u(\mathbf{X})_{t-},\, t\ge 0\}$ is
predictable. Since the function  $u(\omega;\cdot)$ is nearly
Borel, the process $\{u(\mathbf{X}_{t-}),t\ge0\}$ is predictable,
too. Therefore applying the Section Theorem yields
\[
u(\mathbf{X}_{t-})=u(\mathbf{X})_{t-}\,,\quad t\in [0,T_\iota],
\quad P_{m_T}\mbox{-a.s.},
\]
which together with (\ref{eq2.31}) and \cite[Theorem IV.5.29]{MR}
shows  that $u(\omega;\cdot)$ is quasi-continuous.
\end{proof}

\begin{theorem}
\label{tw3.4} Assume that $\varphi,f,g$ satisfy
\mbox{\rm(H1)--(H5)}. Let $N$ be defined by
\mbox{\rm(\ref{eq3.07})} and let $Y$ be the process of Theorem
\ref{stw2.2}. Then for $P$-a.s. $\omega\in\Omega$ the function
$u(\omega;\cdot):E_{0,T}\setminus N\rightarrow\BR$ defined as
\begin{equation}
\label{eq2.7} u(z)=E_z\Big(\varphi(\BBX_{T_\iota})
+\int_0^{T_\iota}f(\BBX_t,Y_t)\,dt+\int_0^{T_\iota}g(\BBX_t,Y_t)\,d^\dagger
\beta_t^\iota\Big)
\end{equation}
is quasi-continuous and for every $z\in E_{0,T}\setminus N$,
\begin{equation}
\label{eq2.3} Y_t=u(\BBX_t),\quad
t\in[0,T_\iota],\quad\BP_z\mbox{\rm-}a.s.
\end{equation}
In particular, $u$ is a solution of \mbox{\rm(\ref{eq4.1})}.
\end{theorem}
\begin{proof}
By \cite[Theorem 4.1.1]{FOT} we may and will assume that $N$ is
properly exceptional. Let $(Y,M)$ be the pair of processes of
Theorem \ref{stw2.2}. By the proof of Theorem {\ref{tw1.1}}, for
every $z\in E_{0,T}\setminus N$  the pair $(Y,M)$ is under $\BP_z$
a limit in $S^2\otimes \MM^2$ of solutions $(Y^n,M^n)$ of some
linear BDSDEs. In particular, $u_n(z)=E_zY_0^n\rightarrow
E_zY_0=u(z)$ for every $z\in E_{0,T}\setminus N$, which when
combined with the fact that $N$ is properly exceptional implies
that $u_n(\BBX_t)\rightarrow u(\BBX_t)$, $t\in[0,T_\iota]$,
$\BP_z$-a.s. for every $z\in E_{0,T}\setminus N$. Since we  know
from Proposition \ref{stw3.41} that (\ref{eq2.3}) holds for
solutions of linear equations, we conclude that (\ref{eq2.3})
holds in the general case. Quasi-continuity of $u$  follows now
from Proposition \ref{stw3.41}. Since  by (\ref{eq2.7}),
(\ref{eq2.3}) and Theorem \ref{stw2.2} conditions (a) (b) of the
definition of a solution SPDE (\ref{eq4.1}) are satisfied, $u$ is
a solution of (\ref{eq4.1}).
\end{proof}

\begin{remark}
\label{rem3.8} Assume (H1)--(H5). \smallskip\\
(i) From Theorems \ref{stw2.2}  and \ref{tw3.4} we know that there
exists a pair $(Y,M)$ (not depending on $z$) of $(\FF_t)$-adapted
c\`adl\`ag processes such that $(Y,M)$ is a solution of
BDSDE${}_z(\varphi,f,g)$ for q.e. $z\in E_{0,T}$ and that the
solution $(Y,M)$ provides  a stochastic representation of  the
solution $u$ of SPDE (\ref{eq4.1}). The representation has the
form (\ref{eq2.7}), or equivalently,
\[
u(z)=Y_0,\quad \BP_z\mbox{\rm-a.s. for q.e. }z\in E_{0,T}.
\]
(ii) On the contrary, if $u$ is a solution of (\ref{eq4.1}) then
there is an $(\FF_t)$-adapted c\`adl\`ag process $M$ (not
depending on $z$) such that for q.e. $z\in E_{0,T}$ the pair
$(u(\BBX),M)$ is a solution of BDSDE${}_z(\varphi,f,g)$. Indeed,
given a solution  $u$ let us define $f_u,g_u$ as
$f_u(z)=f(z,u(z))$, $g_u(z)=g(z,u(z))$, $z\in E_{0,T}$. Let
$(Y,M)$ be the pair of processes of Theorem \ref{stw2.2} such that
$(Y,M)$ is a solution of the linear BDSDE${}_z(\varphi,f_u,g_u)$
for q.e. $z\in E_{0,T}$. By Proposition \ref{stw3.41},
$Y=u(\BBX)$, so the pair $(u(\BBX),M)$ is a solution of
BDSDE${}_z(\varphi,f_u,g_u)$, which means that it is a solution of
BDSDE${}_z(\varphi,f,g)$. Thus starting from $u$ we can construct
a solution $(u(\BBX),M)$ of the BDSDE
\begin{align}
\label{eq3.14} u(\BBX_t)&=\varphi(\BBX_{T_{\iota}})
+\int^{T_{\iota}}_tf(\BBX_r,u(\BBX_r))\,dr\\
&\quad+\int^{T_{\iota}}_0g(\BBX_r,u(\BBX_r))\,d^{\dagger}\beta^{\iota}_r
-\int^{T_{\iota}}_tdM_r,\quad t\in[0,T_{\iota}],\quad
\BP_z\mbox{\rm-a.s.} \nonumber
\end{align}
\end{remark}

\begin{proposition}
\label{stw3.4} Under \mbox{\rm(H1)--(H5)} there exists at most one
solution of \mbox{\rm(\ref{eq4.1})}.
\end{proposition}
\begin{proof}
Let $u_1,u_2$ be two solutions of (\ref{eq4.1}).  By Remark
\ref{rem3.8}, for q.e. $z\in E_{0,T}$ the processes $u_1(\BBX)$,
$u_2(\BBX)$ are the first components of solutions of
BDSDE$_z(\varphi,f,g)$. Therefore by Corollary \ref{wn1.1} and
Lemma \ref{stw2.1}, for q.e. $z\in E_{0,T}$ we have
$u_1(z)=u_2(z)$ $P$-a.s.
\end{proof}
\medskip

In the next section we prove some results on regularity of the
solution $u$ of SPDE (\ref{eq4.1}). Here let us only note that our
proofs are based on equation (\ref{eq3.14}). Clearly
(\ref{eq3.14}) implies that for q.e. $z\in E_{0,T}$,
\begin{equation}
\label{eq3.15} A_t:=u(\BBX_t)-u(\BBX_0)=M_t+N_t,\quad
t\in[0,T_{\iota}],\quad \BP_z\mbox{\rm-a.s.},
\end{equation}
with $N$ defined by (\ref{eq1.06}). We shall see that $M$ is a
random martingale AF  of finite energy and $N$ is a random
continuous AF  of finite,  but in most cases  nonzero energy. This
means that (\ref{eq3.15}) can not be viewed as Fukushima's
decomposition of $A$. Nevertheless, we will prove some estimates
for the energy $e(M)$ of $M$, which when combined with a priori
estimates for $(u(\BBX),M)$ obtained in Proposition \ref{stw1.2}
lead to  energy estimates for $u$.

\section{Regularity of solutions}
\label{sec4}

Let $\HH$ be some Hilbert space  equipped with the norm
$|\cdot|_{\HH}$, let $\BB(\HH)$ denote the Borel $\sigma$-field of
subsets of $\HH$ and let
\begin{itemize}
\item $M^2(0,T;\HH)$ be the space
of all $\FF\otimes\BB([0,T])/\BB(\HH)$ measurable processes
$v:\Omega\times[0,T]\rightarrow\HH$ such that for a.e.
$t\in[0,T]$ the random variable $v(t)$ is
$\FF^{\beta}_{t,T}$-measurable and
\[
E\int^T_0|v(t)|^2_{\HH}\,dt<\infty,
\]
\item $S^2(0,T;\HH)$ be the subspace of $M^2(0,T;\HH)$ consisting of
all processes $v$ such that
\[
\sup_{0\le t\le T}E|v(t)|^2_{\HH}<\infty.
\]
\end{itemize}

Let $u$ be a solution of (\ref{eq4.1}) of Theorem \ref{tw3.4} and
let $u(t)=u(t,\cdot)$. In this section we show that $t\mapsto
u(t)$ belongs to the space $S^2(0,T;H)\cap M^2(0,T;V)$ with $H,V$
defined in Section \ref{sec3.1}, and we prove energy estimates for
$u$, i.e. estimates of $u$ in the norm $\|\cdot\|_{\BB^{0,T}}$
defined as $\|u\|^2_{\BB^{0,T}}=\BB^{0,T}(u,u)$, where
$\BB^{0,T}(u,u)$ is defined by (\ref{eq3.16}). Note that by
assumption (b) in Section \ref{sec3.1}, the norm
$\|\cdot\|_{\BB^{0,T}}$ is equivalent to the usual norm in the
space $L^2(0,T;V)$. We begin with linear equations.

\begin{proposition}
\label{stw2.5} Assume that $f,g$ do not depend on the last
variable $y$ and \mbox{\rm(H1)} is satisfied. Then $u$ defined by
$\mbox{\rm{(\ref{eq2.71})}}$ belongs to $M^2(0,T;V)$ and there is
$c>0$ such that
\[
E\|u\|^2_{\BB^{0,T}}\le c
E\Big(\|\varphi\|^2_{L^2(E;m)}+\|f\|^2_{L^2(E_{0,T};m_T)}
+\sum_{k=1}^{\infty}\|g_k\|^2_{L^2(E_{0,T};m_T)}\Big).
\]
\end{proposition}
\begin{proof}
Put
\[
w(z)=E_z\Big(\varphi(\BBX_{T_\iota})
+\int_0^{T_\iota}f(\BBX_t)\,dt\Big),\quad z\in E_{0,T}.
\]
From  \cite[Theorem 3.7]{K:JFA} and standard energy estimates for
solutions of PDEs it follows  that $w\in M^2(0,T;V)$ and
\[
\|w\|^2_{\BB^{0,T}}\le \|\varphi\|^2_{L^2(E;m)}
+\|f\|^2_{L^2(E_{0,T};m_T)}.
\]
Put
\[
v(z)=E_z\int_0^{T_\iota}g(\BBX_t)\,d^\dagger\beta_t^\iota,\quad
z\in E_{0,T}.
\]
We are going to show that $v\in M^2(0,T;V)$ and
\begin{equation}
\label{eq2.8} E\|v\|^2_{\BB^{0,T}}\le c\sum_{k\ge1}
E\|g_k\|^2_{L^2(E_{0,T};m_T)}.
\end{equation}
To this end, let us first observe that by the stochastic  Fubini
theorem and Markov property,
\begin{align*}
(R_\alpha^{0,T}v)(z)&=E_z\int_0^{T_\iota}e^{-\alpha r}
E_{\BBX_r}\Big( \int_0^{T_\iota}g(\BBX_t)\,d^\dagger
\beta_t^\iota\Big)\,dr\\&= E_z\int_0^{T_\iota}e^{-\alpha
r}E_{z}\Big( \int_r^{T_\iota}g(\BBX_t)\,d^\dagger
\beta_t^\iota|\FF^{\mathbf{X}}_r\Big)\,dr\\&
=E_z\int_0^{T_\iota}e^{-\alpha r}\Big(
\int_r^{T_\iota}g(\BBX_t)\,d^\dagger\beta_t^\iota\Big)\,dr=
\frac1{\alpha} E_z\int_0^{T_\iota}(1-e^{-\alpha t})
g(\BBX_t)\,d^\dagger\beta_t^\iota.
\end{align*}
Hence
\begin{equation}
\label{eq2.9} v(z)-\alpha(R_\alpha^{0,T}v)(z)=E_z\int_0^{T_\iota}
e^{-\alpha t}g(\BBX_t)\,d^\dagger\beta_t^\iota, \quad z\in
E_{0,T}.
\end{equation}
For given $\alpha>0$ and $v,u\in M^2(0,T;V)$ write
\[
\EE^{(\alpha),0,T}(u,v)=\alpha(u-\alpha
R_\alpha^{0,T}u,v)_{L^2(E_{0,T};m_T)}.
\]
By (\ref{eq2.9}),
\[
E\EE^{(\alpha),0,T}(v,v)=\alpha\int_{E_{0,T}}
E\Big(E_z\int_0^{T_\iota}e^{-\alpha t} g(\BBX_t)\,d^\dagger
\beta_t^\iota\cdot E_z\int_0^{T_\iota}g(\BBX_t)\,d^\dagger
\beta_t^\iota\Big)\,m_T(dz).
\]
Using It\^o's isometry and the fact that
$\alpha\hat{R}^{0,T}_\alpha$ is a contraction on
$L^2(E_{0,T};m_T)$ we conclude from the above that
\begin{align}
\label{eq4.3} E\EE^{(\alpha),0,T}(v,v)&\le
\int_{E_{0,T}}E\Big(E_z\int_0^{T_\iota}e^{-\alpha t}
|g(\BBX_t)|^2\,dt\Big)m_T(dz)\\
&= \sum^{\infty}_{k=1}E\|\alpha
R_\alpha^{0,T}g_k^2\|_{L^1(E_{0,T};m_T)} \le \sum^{\infty}_{k=1}
E\| g_k\|^2_{L^2(E_{0,T};m_T)}. \nonumber
\end{align}
By \cite[(6.1.28)]{O3},
\[
E\BB^{0,T}(\alpha R^{0,T}_\alpha v,\alpha R^{0,T}_\alpha v) \le
E\EE^{(\alpha),0,T}(v,v).
\]
When combined with (\ref{eq4.3})  and \cite[Theorem I.2.13]{MR}
this shows that $v\in M^2(0,T;V)$. Therefore letting
$\alpha\rightarrow\infty$ in (\ref{eq4.3}) yields (\ref{eq2.8}).
\end{proof}

\begin{lemma}
\label{lm.port} Let $u:\Omega\times E_{0,T}\rightarrow \BR_+$ and
let $u_\alpha=\alpha R_\alpha^{0,T}u$, $\alpha>0$.
\begin{enumerate}
\item [\rm{(i)}] If $u\in M^2(0,T;H)$ then
$E\|u_\alpha-u\|^2_{L^2(E_{0,T};m_T)}\rightarrow 0$ as
$\alpha\rightarrow\infty$.
\item [\rm{(ii)}] If $u\in M^2(0,T;V)$ then
$E\mathcal{B}^{0,T}(u_\alpha-u,u_\alpha-u)\rightarrow 0$ as
$\alpha\rightarrow\infty$.
\item [\rm{(iii)}] If $u(z)
=E_z\int_0^{T_\iota}g(\mathbf{X}_t)\,d^\dagger\beta_t$ for some
$g$ satisfying \mbox{\rm(H1)} and not depending on $y$, then for
every $t\in [0,T)$, $E\|u_\alpha(t)-u(t)\|^2_{L^2(E;m)}\rightarrow
0$ as $\alpha\rightarrow\infty$.
\end{enumerate}
\end{lemma}
\begin{proof}
(i) Since  $\{ R^{0,T}_\alpha\}$ is a strongly continuous
contraction resolvent  on $L^2(E_{0,T};m_T)$, $u_\alpha\rightarrow
u$ in $L^2(E_{0,T};m_T)$  and $\|u_\alpha\|_{L^2(E_{0,T};m_T)}\le
\|u\|_{L^2(E_{0,T};m_T)}$ for $P$-a.e. $\omega\in \Omega$.
Therefore applying the Lebesgue dominated convergence theorem we
get (i).
\\
(ii) By \cite[Theorem I.2.13]{MR},
$\mathcal{B}^{0,T}(u_\alpha-u,u_\alpha-u)\rightarrow 0$ for
$P$-a.e. $\omega\in \Omega$. Moreover,  by \cite[Lemma
I.2.11]{MR}, there exists $c>0$ (independent of $\omega$) such
that $\mathcal{B}^{0,T}(u_\alpha,u_\alpha)\le c
\mathcal{B}^{0,T}(u,u)$ for $P$-a.e. $\omega\in\Omega$. Therefore
(ii) follows  by the Lebesgue dominated theorem.
\\
(iii) By (\ref{eq2.9}),
\[
E\|u_\alpha(t)-u(t)\|^2_{L^2(E;m)}\le
\mathbb{E}_{t,m}\int_0^{T_\iota}e^{-2\alpha r}
\|g(\mathbf{X}_r)\|^2\,dr,
\]
so it suffices to show that the integral on the right-hand side of
the above inequality is finite for every $t\in [0,T)$. But
\[
\mathbb{E}_{t,m}\int_0^{T_\iota}\|g(\mathbf{X}_r)\|^2\,dr
=EE_{t,m}\int_t^T\sum^{\infty}_{k=1}|g_k(r,X_r)|^2\,dr \le
E\|g\|^2_{L^2(E_{0,T};m_T)},
\]
which implies the desired conclusion.
\end{proof}
\medskip

Let us recall that the energy $e(A)$ of  an AF $A$ of the process
$\BBM^{0,T}$ associated with the form $\EE^{0,T}$ is defined as
\[
e(A)=\frac12\lim_{\alpha\rightarrow\infty}\alpha^2
E_{m_T}\int^{T_{\iota}}_0e^{-\alpha t}A^2_t\,dt,\quad
\]
whenever the integral exists (see, e.g., \cite{O1,O3}). Also note
that if $M$ is a martingale AF of $\BBM^{0,T}$ then the sharp
bracket $\langle M\rangle$ of $M$ is a positive continuous AF of
$\BBM^{0,T}$. Let $\mu_{\langle M\rangle}$ denote the Revuz
measure of $\langle M\rangle$. Then
\begin{equation}
\label{eq4.5} \mu_{\langle M\rangle}(E_{0,T})
=\lim_{t\rightarrow0^+}\frac{1}{t}E_{m_T}\langle M\rangle_t=2e(M)
\end{equation}
(see \cite[Section 5.1.2]{O3}).

In what follows we will be interested in AFs of the form
\begin{equation}
\label{eq4.4} A_t=\tilde u(\mathbf{X}_t)-\tilde
u(\mathbf{X}_0),\quad t\in[0,T_{\iota}],
\end{equation}
where $\tilde{u}$ is a quasi-continuous $m_T$-version of
$u\in\WW_T$. Such AFs admit  a unique decomposition (called
Fukushima's decomposition)
\[
A_t=M^{[u]}_t+N^{[u]}_t,\quad t\in[0,T_{\iota}]
\]
into a martingale AF $M^{[u]}$ of $\BBM^{0,T}$ of finite energy
and a continuous AF $N^{[u]}$ of $\BBM^{0,T}$ of zero energy (see
\cite[Theorem 6.4]{O1}).

\begin{lemma}
\label{lm2.0} Let $k^\beta=\beta(1-\beta\hat{R}_\beta^{0,T}1)\cdot
m_T$.  Then there exists a smooth Radon measure $k$ such that
$k^\beta\rightarrow k$ in the vague topology as
$\beta\rightarrow\infty$, and for every $u\in \WW_T$,
\begin{equation}
\label{eq2.10e} e(M^{[u]})= \|\tilde{u}(0)\|^2_{L^2(E;m)}+
\BB^{0,T}(u,u)-\frac12\int_{E_{0,T}}|\tilde{u}(z)|^2\,k(dz),
\end{equation}
where $\tilde{u}$ is a quasi-continuous $m_T$-version of $u$.
\end{lemma}
\begin{proof}
Let $A$ be defined by (\ref{eq4.4}). By an elementary computation
we get
\begin{align*}
\beta^2 E_{m_T}\int_0^{T_\iota}e^{-\beta t}(u(\mathbf{X}_t)
-u(\mathbf{X}_0))^2\,dt=2\beta (u,u-\beta R^{0,T}_\beta u)
-\beta(u^2,1-\beta\hat{R}_\beta^{0,T}1).
\end{align*}
Hence
\begin{equation}
\label{eqe.1} \beta^2 E_{m_T}\int_0^{T_\iota}e^{-\beta
t}A_t^2\,dt=2\EE^{(\beta),0,T}(u,u)-(u^2,k^\beta).
\end{equation}
By (\ref{eqe.1}),
\[
(u^2,k^\beta)\le 2\EE^{(\beta),0,T}(u,u).
\]
From this we conclude that the sequence $\{k^\beta\}$ is tight in
the vague topology. Therefore if  $u\in\WW_T\cap C_c(E_{T})$ then
letting $\beta\rightarrow\infty$ in the above inequality we get
\begin{equation}
\label{eqe.2} (u^2,k)\le 2\EE^{0,T}(u,u) =\|u(0)\|^2_{L^2(E;m)}
+2\BB^{0,T}(u,u).
\end{equation}
Since it is known that there is a continuous embedding of $\WW_T$
into $C([0,T];H)$, from (\ref{eqe.2}) it follows that there is
$c>0$ such that
\begin{equation}
\label{eq4.6} (u^2,k)\le c\|u\|_{\WW_T}
\end{equation}
for all $u\in \WW_T\cap C_c(E_{T})$. From (\ref{eq4.6}) and
\cite[Theorem 1]{Pierre} it follows that $k$ is a smooth measure.
Furthermore, since each quasi-continuous $u\in \WW_T$ can be
approximated q.e. and in $\WW_T$ by functions from the space
$\WW_T\cap C_c(E_{T})$,  (\ref{eqe.2}) holds true for every
quasi-continuous $u\in \WW_T$. Let $\{u_n\}\subset \WW_T\cap
C_c(E_{T})$ be such that $u_n\rightarrow u$ in $\WW_T$. Then
\begin{align}
\label{eqe.4}  |(u^2,k^{\beta})^{1/2}-(u^2,k)^{1/2}|&\le
|(u^2_n,k^{\beta})^{1/2}-(u^2_n,k)^{1/2}|+\|(u_n-u)(0)\|_{L^2}\\
&\quad+(2\EE^{(\beta),0,T}(u_n-u,u_n-u))^{1/2}\nonumber\\
&\quad+(2\BB^{0,T}(u_n-u,u_n-u))^{1/2}. \nonumber
\end{align}
Since $\WW_T\subset C(0,T;L^2(E;m))$ and the embedding is
continuous, letting $\beta\rightarrow\infty$ and then
$n\rightarrow\infty$ in (\ref{eqe.4}) shows that for every $u\in
\WW_T$,
\[
\int_{E_{0,T}} u^2\,dk^{\beta}\rightarrow \int_{E_{0,T}} u^2\,dk
\]
as $\beta\rightarrow\infty$. By \cite[Theorem 6.4]{O1},
$e(A)=e(M^{[u]})$. Therefore letting $\beta\rightarrow \infty$ in
(\ref{eqe.1}) yields $2e(M^{[u]})=2\EE^{0,T}(u,u)-(u^2,k)$, which
implies (\ref{eq2.10e}).
\end{proof}

\begin{definition}
We say that an $(\FF_t)$-adapted process $A$ is a random additive
functional (random AF for short) of $\BBM^{0,T}$ if for $P$-a.e.
$\omega\in \Omega$ the process $A(\omega;\cdot)$ is an AF of
$\BBM^{0,T}$. Similarly, we say that a process $A$ is a random
martingale (continuous, positive) AF of $\BBM^{0,T}$ if  for
$P$-a.e. $\omega\in \Omega$ the process $A(\omega;\cdot)$ is a
martingale (continuous, positive) AF of $\BBM^{0,T}$.
\end{definition}

\begin{lemma}
\label{lm.200}
Let $\{A^n\}$ be a sequence of random AFs such that
\[
\BE_z\sup_{t\le T_\iota}|A^n_t-A^m_t|\rightarrow0\quad\mbox{for
}m_T\mbox{-a.e. }z\in E_{0,T}\mbox{ as }n,m\rightarrow \infty.
\]
Then there exists a subsequence $(n_k)\subset (n)$ such that for
$P$-a.s. $\omega\in \Omega$,
\[
E_z\sup_{t\le T_\iota}|A^{n_k}_t-A^{n_l}_t|\rightarrow
0\quad\mbox{for }\mbox{q.e. }z\in E_{0,T}\mbox{ as }k,l
\rightarrow \infty.
\]
\end{lemma}
\begin{proof}
Let $\rho\in L^1(E_{0,T};m_T)$ be a strictly positive function
such that $\int_{E_{0,T}}\rho\,dm_T=1$ and
\[
\BE_\nu\sup_{t\le T_\iota}|A^n_t-A^m_t|\rightarrow 0\quad\mbox{as
}n,m\rightarrow \infty,
\]
where $\nu=\rho\cdot m_T$. Let $(n_k)\subset (n)$ be a
subsequence such that
\begin{equation}
\label{eq.ser} \BE_\nu\sup_{t\le T_\iota}
|A^{n_{k+1}}_t-A^{n_k}_t|\le 2^{-k},\quad k\ge1,
\end{equation}
and let $B=\{z\in E_{0,T}:E_z\sup_{t\le
T_\iota}|A^{n_k}_t-A^{n_l}_t|\nrightarrow 0\mbox{ as }
k,l\rightarrow \infty\}$.  Let us stress that $B$ depends on
$\omega$. Write $\tau=\sigma_B$. By the Markov property and
additivity of $A$,
\begin{align*}
P_\nu(\tau<T_\iota)&\le \int_{E_{0,T}}P_z\big(\exists_{s\le
T_\iota} E_z(\sup_{t\le T_\iota\circ\theta_s}|A^{n_l}_t\circ
\theta_s-A^{n_k}_t\circ \theta_s||\FF_s^\BBX)\nrightarrow
0\big)\,d\nu(z)\\&=\int_{E_{0,T}}P_z \big(\exists_{s\le T_\iota}
E_z(\sup_{t+s\le T_\iota}|A^{n_l}_{t+s}-A^{n_l}_s
-A^{n_k}_{t+s}+A^{n_k}_s||\FF_s^\BBX)\nrightarrow
0\big)\,d\nu(z)\\
&\le \int_{E_{0,T}}P_z \big(\exists_{s\le T_\iota} 2E_z(\sup_{t\le
T_\iota}|A^{n_l}_{t}-A^{n_k}_t||\FF_s^\BBX)\nrightarrow
0\big)\,d\nu(z) \\
&\le \int_{E_{0,T}}P_z \big(\sup_{0\le s\le T_\iota}
(E_z(\sup_{t\le T_\iota}|A^{n_l}_{t}-A^{n_k}_t|
|\FF_s^\BBX))^q\nrightarrow 0\big)\,d\nu(z)
\end{align*}
for every $q\in(0,1)$. Let $\Pi=P\otimes(\nu\otimes K)$, where
$K(z,d\omega')=P_z(d\omega')$. By (\ref{eq.ser}) and \cite[Lemma
6.1]{BDHPS},
\begin{align*}
&\sum_{k=1}^{\infty}\int_{\Omega\times E_{0,T}\times\Omega'}
\sup_{s\le T_\iota}(E_z(\sup_{t\le T_\iota}
|A^{n_{k+1}}_{t}-A^{n_k}_t||\FF_s^\BBX))^q
\,d\Pi(\omega,z,\omega')\\&\qquad\le \sum_{k=1}^{\infty}
\frac{1}{1-q}(\BE_\nu\sup_{t\le T_\iota}
|A^{n_{k+1}}_t-A^{n_k}_t|)^q<\infty.
\end{align*}
Since $\Pi(\Omega\times E_{0,T}\times\Omega')=1$, applying the
Borel-Cantelli lemma shows that
\[
\sup_{s\le T_\iota}(E_\cdot(\sup_{t\le
T_\iota}|A^{n_{k}}_{t}-A^{n_l}_t||\FF^{\BBX}_s))^q\rightarrow
0\quad\mbox{as }k,l\rightarrow \infty,\quad \Pi\mbox{-a.e.}
\]
In particular, for $\nu$-a.e. $z\in E_{0,T}$ and $P$-a.e.
$\omega\in\Omega$,
\[
\sup_{s\le T_\iota}(E_z(\sup_{t\le T_\iota}
|A^{n_{k}}_{t}-A^{n_l}_t||\FF^{\BBX}_s))^q\rightarrow
0\quad\mbox{as } k,l\rightarrow \infty,\quad P_z\mbox{-a.e.}
\]
Hence $\BP_\nu(\tau\le T_\iota)=0$, which implies that cap$(B)=0$
$P$-a.e. (see \cite[p. 298]{O2}).
\end{proof}

For a given set $A\subset E_{0,T}\times\Omega$ and
$(z_0,\omega_0)$  we write $A_{z_0}=\{\omega\in\Omega;
(z_0,\omega)\in A\}$ and $A_{\omega_0}=\{z\in E_{0,T};
(z,\omega_0)\in A\}$.

\begin{lemma}
\label{lm.201} Let $A\subset E_{0,T}\times\Omega$ be a measurable
set. If \mbox{\rm cap}$(A_\omega)=0$ for $P$-a.e.
$\omega\in\Omega$, then $ P(A_z)=0$ for \mbox{\rm cap}-q.e. $z\in
E_{0,T}$.
\end{lemma}
\begin{proof}
Let $\nu$ be a smooth bounded measure. Then
$E\int_{E_{0,T}}\mathbf{1}_A(\omega,z)\,\nu(dz)=0$ by the
assumption  on $A$. Therefore using Fubini's theorem we obtain
\[
0=\int_{E_{0,T}}E\mathbf{1}_A(\omega,z)\,\nu(dz)
=\int_{E_{0,T}}P(A_z)\,\nu(dz).
\]
Since $\nu$ was arbitrary, it follows from \cite[Theorem
2.5]{Trutnau} (or remark at the end of Section 4 in \cite{O2})
that $P(A_z)=0$  for cap-q.e. $z\in E_{0,T}$.
\end{proof}

\begin{proposition}
\label{lm4.7} Assume that $u$ is given by \mbox{\rm(\ref{eq2.7})}.
Then $u(\BBX_t)=Y_t$, $t\in[0,T_\iota]$ for q.e. $z\in E_{0,T}$,
where $(Y,M)$ is a solution of $\mbox{\rm{BDSDE}}_z(\varphi,f,g)$.
Moreover,  $M$ is a random martingale AF of $\BBM^{0,T}$ and
\begin{equation}
\label{eq2.10} Ee(M)= E\Big(\|u(0)\|^2_{L^2(E;m)}
+\BB^{0,T}(u,u)-\frac12\int_{E_{0,T}}|u(z)|^2\,k(dz)\Big),
\end{equation}
where $k$ is the killing measure of the form
$(\hat{\EE}^{0,T},D(\hat{\EE}^{0,T}))$ defined in Lemma
\ref{lm2.0}.
\end{proposition}
\begin{proof}
The first assertion  follows from Theorem \ref{tw3.4}. If $g\equiv
0$ then (\ref{eq2.10}) follows from \cite[Lemma 6.1]{O1}.
Therefore we may and will assume that $\varphi\equiv0$,
$f\equiv0$. Let $u_\alpha=\alpha R_\alpha^{0,T}u$. Then
\[
u_\alpha(\BBX_t)=\int_0^t\alpha(u-u_\alpha)(\BBX_r)\,dr
+\int_0^t\,dM^{[u_\alpha]}_r,\quad t\in [0,T_\iota].
\]
By It\^o's formula,
\begin{equation}
\label{eq4.18} \BE_z\int_0^{T_\iota}\,d[M-M^{[u_\alpha]}]_t \le
\BE_z\int_0^{T_\iota}(-2\alpha|u-u_\alpha|^2(\BBX_t)+\|g(\BBX_t)\|^2)\,dt.
\end{equation}
By It\^o's isometry,  for q.e. $z\in E_{0,T}$ we have
\begin{align*}
\alpha\BE_z\int_0^{T_\iota}|u-u_\alpha|^2(\BBX_t)\,dt&
=\alpha\BE_z\int_0^{T_\iota}
\Big(E_{\BBX_t}\int_0^{T_\iota}e^{-\alpha r}
g(\BBX_r)\,d^{\dagger}\beta^\iota_r\Big)^2\,dt\\
&=\alpha\BE_z\int_0^{T_\iota}E_{\mathbf{X}_t}
\int_0^{T_\iota}e^{-2\alpha r} \|g(\BBX_r)\|^2\,dr\,dt\\
&=\frac12 \BE_z\int_0^{T_\iota}2\alpha
R^{0,T}_{2\alpha}\|g\|^2(\BBX_t)\,dt.
\end{align*}
This implies that
\begin{equation}
\label{eq4.12}
2\alpha\BE_z\int_0^{T_\iota}|u-u_\alpha|^2(\BBX_t)\,dt \rightarrow
\mathbb{E}_z\int_0^{T_\iota}\|g(\BBX_t)\|^2\,dt,
\end{equation}
because
\begin{align*}
\BE_z\int_0^{T_\iota}2\alpha R^{0,T}_{2\alpha}\|g\|^2(\BBX_t)\,dt
&=ER^{0,T}(2\alpha R^{0,T}_{2\alpha}\|g\|^2)(z)
=ER^{0,T}\|g\|^2(z)-R^{0,T}_{2\alpha}\|g\|^2(z)\\
&=\mathbb{E}_z\int_0^{T_\iota}\|g(\BBX_t)\|^2\,dt
-\mathbb{E}_z\int_0^{T_\iota}e^{-2\alpha t}\|g(\BBX_t)\|^2\,dt
\end{align*}
and by Lemma \ref{stw2.1},
$\mathbb{E}_z\int_0^{T_\iota}\|g(\BBX_t)\|^2\,dt<\infty$ for q.e.
$z\in E_{0,T}$. By (\ref{eq4.18}) and (\ref{eq4.12}),
\begin{equation}
\label{eqe.3}
\mathbb{E}_z\int_0^{T_\iota}\,d[M-M^{[u_\alpha]}]_t\rightarrow 0
\end{equation}
for q.e. $z\in E_{0,T}$. From (\ref{eqe.3}),
Burkholder-Davis-Gundy inequality and Lemma \ref{lm.200} we
conclude that there exists a process $\tilde{M}$ such that, up to
a subsequence, for $P$-a.e. $\omega\in \Omega$,
\begin{equation}
\label{eq.300} E_z\sup_{t\le
T_\iota}|M^{[u_\alpha]}_t-\tilde{M}_t|\rightarrow 0
\end{equation}
for q.e. $z\in E_{0,T}$. By a standard argument (see  the
reasoning following Eq. (5.2.23) in the proof of \cite[Theorem
5.2.1]{FOT}), $\tilde{M}(\omega;\cdot)$ is a martingale AF of
$\BBM^{0,T}$ for $P$-a.e. $\omega\in\Omega$. By Lemma
\ref{lm.201}, (\ref{eqe.3}) and (\ref{eq.300}), $\tilde{M}$ is a
$\BP_z$-modification of $M$ for q.e. $z\in E_{0,T}$. By
Proposition \ref{stw2.5} and Lemma \ref{lm.port},
$u_\alpha\rightarrow u$ in $M^2(0,T;V)$ and
$E|u_{\alpha}(t)-u(t)|^2_{L^2(E;m)}\rightarrow0$ for every $t\in
[0,T)$ as $\alpha\rightarrow\infty$. Write
$M^\alpha=M^{[u_\alpha]}$. Let $\alpha_n\rightarrow\infty$. By
Lemma \ref{lm2.0},
\[
Ee(M^{\alpha_n}-M^{\alpha_m})\le
E\|(u_{\alpha_n}-u_{\alpha_m})(0)\|^2_{L^2(E;m)}
+E\BB^{0,T}(u_{\alpha_n}-u_{\alpha_m}, u_{\alpha_n}-u_{\alpha_m}),
\]
which converges to zero as $n,m\rightarrow\infty$.  From this and
(\ref{eqe.3}) it follows that
\[
\lim_{n\rightarrow\infty}Ee(M^{\alpha_n})=Ee(M).
\]
By Lemma {\ref{lm2.0}},
\begin{equation}
\label{eq2.13} Ee(M^\alpha)
=E\|u_\alpha(0)\|^2_{L_2(E;m)}+E\BB^{0,T}(u_\alpha,u_\alpha)
-\frac12E\int_{E_{0,T}}|u_{\alpha}(z)|^2\,k(dz).
\end{equation}
Observe now that
\begin{equation}
\label{eq4.16} E\int_{E_{0,T}}|u_{\alpha}(z)|^2\,k(dz)\rightarrow
E\int_{E_{0,T}}|u(z)|^2\,k(dz).
\end{equation}
Indeed, since from the proof of Lemma \ref{lm2.0} we know that
(\ref{eqe.2}) holds for quasi-continuous elements of $\WW_T$,
there exists $v\in L^2(E_{0,T}; k)$ such that
$(E|u_\alpha|^2)^{1/2}\rightarrow v$ in $L^2(E_{0,T};k)$.
Therefore the proof of  (\ref{eq4.16}) is completed by showing
that $v^2=Eu^2$ $k$-a.e. By (\ref{eq2.9}),
\[
E|u(z)-u_\alpha(z)|^2\le \mathbb{E}_z\int_0^{T_\iota}e^{-2\alpha
t}\|g(\mathbf{X}_t\|^2\,dt,\quad z\in E_{0,T}.
\]
By Lemma \ref{stw2.1},
$\mathbb{E}_z\int_0^{T_\iota}\|g(\mathbf{X}_t)\|^2\,dt<\infty$ for
q.e. $z\in E_{0,T}$. Hence $Eu^2_\alpha\rightarrow Eu^2$ q.e.
Consequently, $Eu^2_\alpha\rightarrow Eu^2$ $k$-a.e.  since $k$ is
smooth. Thus $v^2=Eu^2$, which completes the proof of
(\ref{eq4.16}). Letting $\alpha\rightarrow\infty$ in
(\ref{eq2.13}) and using (\ref{eq4.16}) we get (\ref{eq2.10}).
\end{proof}
\begin{remark}
Let $N$ be defined by (\ref{eq1.06}). A direct calculation shows that
\begin{equation}
\label{eqeqeq1}
e(N)=\frac12 \|g_u\|^2_{L^2(E_{0,T}; m_T)}.
\end{equation}
From this one can conclude that $u\notin \mathbb{W}$, where
$\mathbb{W}$ is defined by (\ref{eq1.10}). Indeed, suppose that
$u\in\mathbb{W}$. We may assume that $\varphi=0, f=0$ and $g$ does
not depend on $y$. Let $A^{[u]}$ be defined by (\ref{eq1.05}) and
$u_\alpha$ be as in the proof of Proposition \ref{lm4.7}. Then by
the proof of \cite[Theorem 5.2.2]{FOT} (see also \cite[Theorem
4.5]{Trutnau}), $N=N^{[u]}$ and $N^{[u_\alpha]}$ converges to $N$
uniformly on compacts in probability $\mathbb{P}_{m_T}$. Since
$N=N^{[u]}$ we have
\begin{align*}
Ee(N)&\le 2Ee(N-N^{[u_\alpha]})+2Ee(A^{[u_\alpha]})
=2Ee(N-N^{[u_\alpha]})\\&\le 4Ee(A^{[u]}-A^{[u_\alpha]})
+4Ee(M^{[u]}-M^{[u_\alpha]}).
\end{align*}
Since $u\in \mathbb{W}$, we have
\begin{align*}
Ee(A^{[u]}-A^{[u_\alpha]})&\le E\EE(u_\alpha-u,u_\alpha-u)\\
& =E\|u_{\alpha}(0)-u(0)\|^2_{L^2(E;m)} +E\BB^{0,T}(u_{\alpha}-u,
u_{\alpha}-u).
\end{align*}
(see, e.g., \cite[Eq. (13)]{Trutnau}). When combined with the
previous inequality and Lemma \ref{lm2.0} this shows that
$e(N)=0$, which contradicts (\ref{eqeqeq1}).
\end{remark}

\medskip

In what follows by $\MM^+$ we denote the set of all positive Borel
measures on $E_{0,T}$. By $\MM^+_{0,b}$ we denote the subset of
$\MM^+$ consisting of all bounded measures  which charge no set of
zero capacity associated with the form $(\EE^{0,T},D(\EE^{0,T}))$.
The total variation norm of $\mu\in\MM^+_{0,b}$ will be denoted by
$\|\mu\|_{TV}$. To shorten notation, for $\mu\in\MM^+$ we write
$P_{\mu}(\cdot)=\int_{E_{0,T}}P_z(\cdot)\,\mu(dz)$ and by
$E_{\mu}$ (resp. $\BE_{\mu}$) we denote the expectation with
respect to $P_{\mu}$ (resp. $P\otimes P_{\mu}$).

\begin{definition}
We say that $\mu:\Omega\times\BB(E_{0,T})\rightarrow\BR$ is a
random measure if $\mu(\omega;\cdot)$ is a positive Radon measure
on $E_{0,T}$ for $P$-a.s. $\omega\in\Omega$ and $\mu(\omega;B)$ is
$\FF^{\beta}_{T}$-measurable for every $B\in\BB(E_{0,T})$.
\end{definition}

Given a random measure $\mu$ such that
$\mu(\omega;\cdot)\in\MM^+_{0,b}$ for $P$-a.e. $\omega\in\Omega$
we denote by $A^{\mu}$  the random positive AF of $\BBM^{0,T}$
such that for $P$-a.e. $\omega\in\Omega$ the measure
$\mu(\omega;\cdot)$ is the Revuz measure of the AF
$A(\omega;\cdot)$. We call $A^{\mu}$ the random AF associated with
$\mu$. Note that if the random AF associated with $\mu$ exists,
then it is uniquely determined.

In the rest of the paper writing $A^{\mu}$ for some random measure
$\mu$ we tacitly assume  that this random AF exists.

\begin{lemma}
\label{lm2.2} Assume that $\mu,\nu$ are random measures such that
$\mu,\nu\in\MM_{0,b}^+(E_{0,T})$, $P$-a.s. If for some $v\in
M^2(0,T;V)\cap C([0,T];L^2(\Omega\times E;P\otimes m))$,
\[
Ev(z)+\BE_z\int_0^{T_\iota}\,dA_t^{\mu}
\le\BE_z\int_0^{T_\iota}\,dA_t^\nu
\]
for $m_T$-a.e. $z\in E_{0,T}$ then
\begin{equation}
\label{eqe.6} E\int_{[0,T)\times E}v(z)k(dz)+E\|\mu\|_{TV}\le
E\|\nu\|_{TV}.
\end{equation}
\end{lemma}

\begin{proof}
Since $Ev(z)+ER^{0,T}\mu(z)\le ER^{0,T}\nu(z)$ for $m_T$-a.e.
$z\in E_{0,T}$, we have
\begin{align*}
&E(v,\alpha(1-\alpha \hat{R}_\alpha^{0,T}1))_{L^2(E_{0,T};m_T)}+
E(R^{0,T}\mu,\alpha(1-\alpha \hat{R}_\alpha^{0,T}1))_{L^2(E_{0,T};m_T)}\\
&\quad\le E(R^{0,T}\nu,\alpha(1-\alpha
\hat{R}_\alpha^{0,T}1))_{L^2(E_{0,T};m_T)}.
\end{align*}
Hence
\begin{align*}
&E(v,\alpha(1-\alpha \hat{R}_\alpha^{0,T}1))_{L^2(E_{0,T};m_T)}
+E(\alpha(I-\alpha R^{0,T}_\alpha)R^{0,T}\mu,1)_{L^2(E_{0,T};m_T)}
\\&\quad\le E(\alpha(I-\alpha R^{0,T}_{\alpha})
R^{0,T}\nu,1)_{L^2(E_{0,T};m_T)}.
\end{align*}
It is an elementary check that
\[
\alpha(I-\alpha R^{0,T}_\alpha)R^{0,T}\mu(z)
=E_z\int_0^{T_\iota}\alpha e^{-\alpha t}\, dA_t^\mu,\quad z\in
E_{0,T}.
\]
Therefore
\begin{equation}
\label{eqe.5}
E(v,\alpha(1-\alpha\hat{R}_\alpha^{0,T}1))_{L^2(E_{0,T};m_T)}
+\BE_{m_T}\int_0^{T_\iota}\alpha e^{-\alpha t}\,
dA_t^\mu\le\BE_{m_T}\int_0^{T_\iota}\alpha e^{-\alpha t}\,
dA_t^\nu.
\end{equation}
Letting $\alpha\rightarrow\infty$ in (\ref{eqe.5}) we get
(\ref{eqe.6}). Indeed, directly from the definition of the Revuz
duality it follows that the integrals involving $A^{\mu}$ and
$A^{\nu}$ converge to $E\|\mu\|_{TV}$ and $E\|\nu\|_{TV}$,
respectively. To show the convergence of the first term on the
left-hand side of (\ref{eqe.5}), let us first assume that
$v(T)=0$. Set $v_m=\beta_m R^{0,T}_{\beta_m}v$. Then by
(\ref{eqe.4}),
\begin{align*}
\limsup _{n\rightarrow \infty}
|(v^2,k^{\alpha_n})^{1/2}-(v^2,k)^{1/2}|
&\le4\|(v_m-v,v_m-v)(0)\|_{L_2(E;m)}\\
&\quad+4(\BB^{0,T}(v_m-v,v_m-v))^{1/2}.
\end{align*}
Since $v_m(t)\rightarrow v(t)$ in $L^2(E;m)$ for every $t\in
[0,T)$ and $v_m\rightarrow v$ in $M^2(0,T;V)$, this shows  the
desired convergence of the first term.  In general, if $v(T)\neq
0$, we consider a sequence $\{v_n\}\subset M^2(0,T;V)\cap
C([0,T];L^2(\Omega\times E;P\otimes m))$ such that $v_n\le v$,
$v_n(T)=0$ and $v_n\nearrow v$ on $[0,T)\times E$. Of course,
$v_n$ satisfies the assumptions of the lemma, so by what has
already been proved, (\ref{eqe.6}) is satisfied with $v$ replaced
by $v_n$. Letting $n\rightarrow\infty$ we get (\ref{eqe.6}) for
$v$.
\end{proof}

\begin{theorem}
\label{th4.4} Assume \mbox{\rm(H1)--(H3)}. Let $u$ be a solution
of \mbox{\rm SPDE (\ref{eq4.1})}. Then $u\in S^2(0,T;L^2(E;m))\cap
M^2(0,T;V)$ and there is $c>0$ depending only on $T,L,l$ such that
\begin{align}
\label{eq4.19}
&\sup_{0\le t\le T}E\|u(t)\|^2_{L^2(E;m)}+E\BB^{0,T}(u,u)\\
&\qquad\le cE\Big(\|\varphi\|^2_{L^2(E;m)}+
\|f(\cdot,0)\|^2_{L^2(E_{0,T};m_T)}+\sum^{\infty}_{k=1}
\|g_k(\cdot,0)\|^2_{L^2(E_{0,T};m_T)}\Big). \nonumber
\end{align}
\end{theorem}
\begin{proof}
By Proposition \ref{stw1.2},
\begin{align*}
&\BE_z\sup_{0\le t\le T_\iota}
|u(\BBX_t)|^2+\BE_z\int_0^{T_\iota}d[M]_t \\
&\qquad\le c\BE_z\Big(|\varphi(\BBX_{T_\iota})|^2 +
\int_0^{T_\iota}(|f(\BBX_t,0)|^2+\|g(\BBX_t,0)\|^2)\,dt\Big).
\end{align*}
Hence
\[
E|u(z)|^2 +\BE_z\int_0^{T_\iota}d\langle M\rangle_t \le
c\BE_z\int_0^{T_\iota}dA^{\nu}_t,
\]
where
\[
\nu(dz)=(\delta_{\{T\}}\otimes|\varphi|^2\cdot m)(dz)+
(|f(z,0)|^2+\|g(z,0)\|^2)\,m_T(dz).
\]
Let $\mu_{\langle M\rangle}$ denote the random smooth measure
associated with the random continuous AF $\langle M\rangle$ of
$\BBM$. By Lemma \ref{lm2.2},
\begin{equation}
\label{eq2.16} E\int_{E_{0,T}}|u(z)|^2\,k(dz) +E\|\mu_{\langle
M\rangle} \|_{TV}\le cE\|\nu\|_{TV}.
\end{equation}
Since by (\ref{eq4.5}), $\|\mu_{\langle M\rangle}\|_{TV}=2e(M)$,
$P$-a.s., it follows  from Proposition \ref{lm4.7} and
(\ref{eq2.16}) that
\[
E\|u(0)\|^2_{L^2(E;m)}+E\BB^{0,T}(u,u)\le cE\|\nu\|_{TV}.
\]
Since the same estimate can be obtained on any interval $[t,T]$
with  $t\in(0,T)$, and $cE\|\nu\|_{TV}$ is equal to the right-hand
side of (\ref{eq4.19}), the theorem is proved.
\end{proof}

\section{BDSDEs with Brownian filtration}
\label{sec5}

In the present section and in Section \ref{sec6} we assume that
the filtration $(\GG_t)$ of Section \ref{sec2}  is generated by a
$d$-dimensional Wiener process $W$ on $\Omega'$. This will allow
us to treat in Section \ref{sec6} equations  dependent on the
gradient of a solution.

We also assume that we are given an $\FF_T$-measurable random
variable $\xi$ and two families
$\{f(t,y,z),t\ge0\}_{y\in\BR,z\in\BR^d}$,
$\{g_k(t,y,z),t\ge0\}_{y\in\BR,z\in\BR^d,k\in\BN}$ of processes of
class $M$ (as in Section \ref{sec2}, in our notation we omit the
dependence on $(\omega,\omega')\in\Omega\times\Omega'$). We set
$g(\cdot,y,z)=(g_1(\cdot,y,z),g_2(\cdot,y,z),\dots)$.

Let us consider the following hypotheses.

\begin{enumerate}
\item[(B1)] $\BE|\xi|^2+\mathbb{E}\int_0^T|f(t,0,0)|^2\,dt
+\mathbb{E}\int_0^T\|g(t,0,0)\|^2\,dt<\infty$.

\item[(B2)]$\int_0^T|f(t,y,0)|\,dt<\infty$,
$\mathbb{P}$-a.s. for every $y\in\BR$.

\item[(B3)]There exist $l,L>0$, $m\in (0,1)$
and $M_k, L_k:\Omega\times[0,T]\rightarrow\mathbb{R}_+$ such that
$\sup_{0\le t\le T}\sum_{k=1}^{\infty}L^2_k(t)\le l$, $\sup_{0\le
t\le T} \sum_{k=1}^{\infty}M^2_k(t)\le m$, $P$-a.s. and for a.e.
$t\in[0,T]$,
\begin{enumerate}
\item[(a)] $(f(t,y,z)-f(t,y',z))(y-y')\le L|y-y'|^2$ for all
$y,y'\in\BR$, $z\in\BR^d$,
\item[(b)] $|f(t,y,z)-f(t,y,z')|\le L|z-z'|$ for all $y\in\BR$,
$z,z'\in\BR^d$,
\item[(c)] $|g_k(t,y,z)-g_k(t,y',z')|\le L_k(t)|y-y'|
+M_k(t)|z-z'|$ for all $y,y'\in\BR$, $z,z'\in\BR^d$.
\end{enumerate}
\item[(B4)]For a.e. $t\in[0,T]$ and every $z\in\BR^d$ the mapping
$\BR\ni y\mapsto f(t,y,z)$
is continuous.
\end{enumerate}

\begin{definition}
We say that a pair $(Y,Z)\in\mathcal{S}^2\times M^2$ is a solution
of BDSDE$(\xi,f,g)$ if
\begin{enumerate}
\item[(a)]
$\BP(\int_0^T(|f(t,Y_t,0)|+\|g(t,Y_t,0)\|^2)\,dt<\infty)=1$,

\item[(b)] $Y_t=\xi+\int_t^T f(r,Y_r,Z_r)\,dr
+\int_t^Tg(r,Y_r,Z_r)\,d^\dagger\beta_r-\int_t^TZ_r\,dW_r$, $t\in
[0,T]$, $\BP$-a.s.
\end{enumerate}
\end{definition}

\begin{remark}
If the coefficients $f,g$ do not depend on $z$ and $(Y,Z)$ is a
solution of BDSDE$(\xi,f,g)$  then the pair
\[
(Y_t,M_t):=(Y_t,\int_0^t Z_r\,dW_r),\quad t\in[0,T]
\]
is a solution of BDSDE$(\xi,f,g)$ in the sense of Section
\ref{sec3}. The difference between Section \ref{sec3} and Section
\ref{sec5} is that in the present section we have additional
information on the filtration $(\GG_t)$, which  gives us
additional information on the component $M$ of the solution.
\end{remark}

\begin{proposition}
\label{stw1.111} Let $g$  satisfy $\mbox{\rm{(B3c)}}$ and either
$f$ or $f'$ satisfy $\mbox{\rm{(B3a), (B3b)}}$.  If $\xi\le\xi'$,
$\mathbb{P}$-a.s. and $f'(t,y,z)\le f(t,y,z)$ for a.e. $t\in[0,T]$
and every $y\in\BR$ and $z\in\BR^d$ then
\[
Y'_t\le Y_t,\quad t\in[0,T],\quad\BP\mbox{\rm{-}}a.s.,
\]
where $(Y,M)$, $(Y',M')$ are  solutions of
$\mbox{\rm{BDSDE}}(\xi,f,g)$ and $\mbox{\rm{BSDE}}(\xi',f',g)$,
respectively.
\end{proposition}
\begin{proof}
Assume that $f$ satisfies (B3a), (B3b). By  It\^o's formula,
\begin{align*}
&|(Y_t'-Y_t)^+|^2+\int_t^T\mathbf{1}_{\{Y'_r>Y_r\}}|Z'_r-Z_r|^2\,dr\\
&\quad\le 2\int_t^T(Y'_r-Y_r)^+(f'(r,Y_r',Z'_r)-f(r,Y_r,Z_r))\,dr
-2\int_t^T(Y'_{r}-Y_{r})^+(Z'_r-Z_r)\,dW_r\\
&\qquad+\sum^{\infty}_{k=1}
\int_t^T\mathbf{1}_{\{Y'_r>Y_r\}}|g_k(r,Y'_r,Z'_r)-g_k(r,Y_r,Z_r)|^2\,dr.
\end{align*}
By the assumptions,
\[
\int_t^T(Y'_r-Y_r)^+(f'(r,Y'_r,Z'_r)-f(r,Y_r,Z_r))\,dr
\le 2L\int_t^T|(Y'_r-Y_r)^+|^2\,dr
\]
and
\begin{align*}
&\sum^{\infty}_{k=1}\int_t^T
\mathbf{1}_{\{Y'_r>Y_r\}}|g_k(r,Y'_r,Z'_r)-g_k(r,Y_r,Z_r)|^2\,dr\\
&\quad\le
\sum^{\infty}_{k=1}\int_t^T(L^2_k(r)\mathbf{1}_{\{Y'_r>Y_r\}}|Y'_r-Y_r|^2
+M^2_k(r)\mathbf{1}_{\{Y'_r>Y_r\}}|Z'_r-Z_r|^2)\,dr \\
&\quad\le l\int_t^T|(Y'_r-Y_r)^+|^2\,dr
+m\int_t^T\mathbf{1}_{\{Y'_r>Y_r\}}|Z'_r-Z_r|^2\,dr.
\end{align*}
By the above estimates,
\[
\mathbb{E}|(Y'_t-Y_t)^+|^2
\le2(L+l)\mathbb{E}\int_t^T|(Y'_r-Y_r)^+|^2\,dr,
\quad t\in[0,T],
\]
so  Gronwall's lemma yields the desired result.
\end{proof}

\begin{corollary}
\label{cor5.3} Let assumption \mbox{\rm{(B3)}} hold. Then
there exists at most one solution of \mbox{\rm BDSDE}$(\xi,f,g)$.
\end{corollary}

\begin{theorem}
\label{tw4.1} Let assumptions \mbox{\rm{(B1)--(B4)}} hold.  Then
there exists a solution $(Y,Z)$ of \mbox{\rm BDSDE}$(\xi,f,g)$.
\end{theorem}
\begin{proof}
Let $M^2_k$ denote the set of $k$-dimensional processes
$X=(X^1,\dots,X^k)$ such that $X^i\in M^2$, $i=1,\dots,k$. Define
the mapping $\Phi: M^2_1\otimes M^2_d\rightarrow M^2_1\otimes
M^2_d$ by letting $\Phi(U,V)$ be the solution $(Y,Z)$ of the BSDE
\[
Y_t=\xi+\int_t^T f(r,Y_r,V_r)\,dr+\int_t^T g(r,Y_r,V_r)\,d^\dagger
B_r -\int_t^T Z_r\,dW_r,\quad t\in [0,T].
\]
From Corollary \ref{cor5.3}, Theorem \ref{tw1.1} and the
representation theorem for Brownian filtration it follows that
$\Phi$ is well defined. Let $(U^1,V^1), (U^2,V^2)\in M^2_1\otimes
M^2_d$ and $(Y^i,Z^i)=\Phi(U^i,V^i)$, $i=1,2$. By It\^o's formula,
\begin{align*}
&\mathbb{E}e^{\beta t}|Y^1_t-Y^2_t|^2+\beta \mathbb{E}\int_t^Te^{\beta r}
|Y^1_r-Y^2_r|^2\,dr+\mathbb{E}\int_t^Te^{\beta r}|Z^1_r-Z^2_r|^2\,dr\\
&\qquad= 2\mathbb{E}\int_t^Te^{\beta r}
(f(r,Y_r^1,V^1_r)-f(r,Y^2_r,V^2_r))(Y^1_r-Y^2_r)\,dr\\
&\qquad\quad+\mathbb{E}\int_t^Te^{\beta r}\|g(r,Y^1_r,V^1_r)
-g(r,Y^2_r,V^2_r)\|^2\,dr\\
&\qquad\le 2L\mathbb{E}\int_t^Te^{\beta
r}(|Y^1_r-Y^2_r|^2+|V^1_r-V^2_r|\cdot|Y^1_r-Y^2_r|)\,dr\\
&\qquad\quad+ l\mathbb{E}\int_t^Te^{\beta
r}|Y^1_r-Y^2_r|^2\,dr+mE\int_t^T|V^1_r-V^2_r|^2\,dr.
\end{align*}
Hence, for every $\alpha>0$,
\[
c \mathbb{E}\int_t^Te^{\beta r}|Y^1_r-Y^2_r|^2\,dr
+\mathbb{E}\int_t^T e^{\beta r}|Z^1_r-Z^2_r|^2\,dr \le
(\alpha+m)\mathbb{E}\int_t^T|V^1_r-V^2_r|^2\,dr
\]
with $c=\beta-2L-(2L)^2\alpha^{-1}-l>0$. Let $\alpha, \beta$ be
chosen so that $c>0$ and $\alpha+m<1$. Then $\Phi$ is a
contraction if we equip   $M^2_1\times M^2_d$ with the norm
\begin{equation}
\label{eq.norm} \|(Y,Z)\|^2_\beta=\mathbb{E}\int_0^Te^{\beta r}
(c|Y_r|^2+|Z_r|^2)\,dr.
\end{equation}
By Banach's principle,  $\Phi$ has a fixed point. Of course, it
solves BDSDE$(\xi,f,g)$.
\end{proof}

\section{SPDEs with divergence form operator}
\label{sec6}

In this section we consider equations of the form (\ref{eq0.0})
with $A$ being a uniformly elliptic divergence form operator. We
allow, however, the coefficients $f,g$ to depend on the gradient
of a solution. More precisely, we assume that $E=D$ is an nonempty
bounded open subset of $\BR^d$ and
\[
B^{(t)}(\varphi,\psi)=\sum^d_{i,j=1}\int_D
a_{ij}(t,x)\varphi_{x_i}(x)\psi_{x_j}(x)\,dx,\quad \varphi,\psi\in
V=H^1_0(D),
\]
where $a_{ij}:[0,T]\times D\rightarrow\BR$ are measurable
functions such that for every $(t,x)\in[0,T]\times D$,
\[
a_{ij}=a_{ji},\quad \lambda|\xi|^2\le\sum_{i,j=1}^d
a_{ij}(t,x)\xi_i\xi_j\le\Lambda |\xi|^2,\quad \xi\in\BR^d
\]
for some $0<\lambda\le\Lambda$. In this case the operator $A_t$
associated with $(B^{(t)},V)$ is given by (\ref{eq6.01}). Suppose
we are given measurable functions $\varphi:D\rightarrow\BR$ and
$f,g_k:\Omega\times[0,T]\times D\times \BR\times\BR^d\rightarrow
\BR$. We consider equation  of the form
\begin{equation}
\label{eq6.1} du(t)=-(A_tu+f(t,x,u,\sigma\nabla u))\,dt
-g(t,x,u,\sigma\nabla u)\,d^{\dagger}\beta_t,\quad u(T)=\varphi,
\end{equation}
where $\sigma$ is such that $\sigma\cdot\sigma^T=a$. We are going
to show that (\ref{eq6.1}) has a unique solution under the
following assumptions:
\begin{enumerate}
\item[(D1)] $E\|\varphi\|^2_{L^2(E;m)}
+E\|f(\cdot,0,0)\|^2_{L^2(E_{0,T};m_T)}
+E\sum_{k=1}^\infty\|g_k(\cdot,0,0)\|^2_{L^2(E_{0,T};m_T)}<\infty$.

\item[(D2)] For all $y\in\BR$ and $e\in\BR^d$ the mapping
$E_{0,T}\ni z\mapsto f(z,y,e)$ belongs to $q\mathbb{L}^1$.

\item[(D3)]There exist $l,L>0$, $m\in (0,1)$
and functions $M_k, L_k:E_{0,T}\rightarrow\mathbb{R}_+$ such that
$\sup_{z\in E_{0,T}}\sum_kL_k(r)\le l$, $\sup_{z\in E_{0,T}}
\sum_kM_k(r)\le m$ and for every $z\in E_{0,T}$ we have
\begin{enumerate}
\item[(a)] $(f(z,y,e)-f(z,y',e))(y-y')\le L|y-y'|^2$ for all
$y,y'\in\BR$, $e\in\BR^d$,
\item[(b)]$|f(z,y,e)-f(z,y,e')|\le L|e-e'|$ for all $y\in\BR$,
$e,e'\in\BR^d$,
\item[(c)] $|g_k(z,y,e)-g_k(z,y',e')|\le L_k(z)|y-y'|
+M_k(z)|e-e'|$ for all $y,y'\in\BR$, $e,e'\in\BR^d$.
\end{enumerate}
\item[(D4)]For every $z\in E_{0,T}$ and $e\in\BR^d$ the  mapping
$\BR\ni y\mapsto f(z,y,e)$ is continuous.
\item[(D5)] For every $y\in\BR$, $e\in\BR^d$ the mappings
$f(\cdot,y,e),g_k(\cdot,y,e):\Omega\times E_T\rightarrow\BR$,
$k\in\BN$, are $(\FF^{\beta}_{t,T})$-progressively measurable.
\end{enumerate}

The process $\BBM^{0,T}$ associated with the operator
$\frac{\partial}{\partial t}-A_t$ has the following unique
Fukushima decomposition
\[
\BBX_t=\BBX_0+\BBM_t+\mathbf{A}_t, \quad t\in[0,T_{\iota}],\quad
P_z\mbox{-a.s.},\quad z\in E_{0,T},
\]
where $\BBM$ is a martingale AF of $\BBM^{0,T}$ of finite energy
and $\mathbf{A}$ is a continuous AF of $\BBM^{0,T}$ of zero
energy. It is well known that $W$ defined as
\[
W_t=\int_0^t\sigma^{-1}(\BBX_r)\,d\BBM^\pi_r,\quad t\ge0,
\]
where $\BBM^\pi =\pi(\BBM)$ and
$\pi(x_1,\dots,x_{d+1})=(x_2,\dots,x_{d+1})$ for  $x_i\in\BR$,
$i=1,\dots,d+1,$ is a standard $(\GG_t)$-Brownian motion under
$P_z$.

\begin{proposition}
\label{stw5.1} Assume that $\varphi,f,g$ satisfy
\mbox{\rm(D1)--(D5)} and $f,g$ do not depend on the last variable
$e$. Let $(Y,M)$ be the pair of process of Theorem \ref{stw2.2}
and let $u$ be the function defined by \mbox{\rm(\ref{eq2.7})}.
Then $u\in M^2(0,T; H_0^1(D))$ and for q.e. $z\in E_{0,T}$,
\[
\Big(u(\BBX_t),\int_0^t\sigma\nabla
u(\BBX_r)\,dW_r\Big)=(Y_t,M_t),\quad t\in [0,T_\iota],\quad
\BP_z\mbox{-a.s.}
\]
\end{proposition}
\begin{proof}
That $Y_t=u(\BBX_t),t\in [0,T_\iota]$,  follows from Theorem
\ref{tw3.4}. Observe that $M=M^1+M^2$, where $M^1, M^2$ are
processes of Theorem \ref{stw2.2} associated with the data
$(\varphi,f(\cdot,u),0)$ and $(0,0,g(\cdot,u))$, respectively. The
desired representation for $M^1$ follows from \cite[Proposition
3.6]{K:SPA}. Therefore we may and will assume that $\varphi=0,
f=0$ and $g$ does not depend on $y$. By Theorem \ref{th4.4}, $u\in
M^2(0,T; H^1_0(D))$. Let $u_\alpha$ be defined as in the proof of
Proposition \ref{lm4.7}. Then by (\ref{eqe.3}) and the
Burkholder-Davis-Gundy inequality,
\begin{equation}
\label{eq6.6} \lim_{\alpha\rightarrow\infty}\BE_z\sup_{0\le t\le
T_\iota} |M_t-M^{[u_\alpha]}_t|^2=0.
\end{equation}
On the other hand, it is well known (see \cite[Theorem
5.6]{Trutnau})  that
\[
M^{[u_\alpha]}_t=\int_0^t\sigma\nabla u_\alpha(\BBX_r)\,dW_r,
\quad t\in [0,T_\iota].
\]
Hence
\begin{align*}
\BE_\nu\sup_{0\le t\le
T_\iota}|M^{[u_\alpha]}_t-\int_0^t\sigma\nabla
u(\BBX_r)\,dW_r|^2&\le4\Lambda\BE_\nu\int_0^{T_\iota}|\nabla
u-\nabla u_\alpha|^2(\BBX_t)\,dt\\&\le 4\Lambda E\|\nabla u-\nabla
u_\alpha\|^2_{L^2(E_{0,T};m_T)}\|\hat{R}^{0,T}\nu\|_\infty
\end{align*}
for every $\nu\in \hat{S}_{00}$ (for the definition of $\hat
S_{00}$ see Remark \ref{rem3.1}). Since $u_\alpha\rightarrow u$ in
$M^2(0,T;H^1_0(D))$ (see Lemma \ref{lm.port}), using standard
arguments (see the reasoning in the proof of \cite[Theorem
5.2.1]{FOT}) we conclude from the above inequality and
\cite[Theorem 2.5]{Trutnau} (see also the remark at the end of
Section 4 in \cite{O2}) that, up to a subsequence,
\[
\lim_{\alpha\rightarrow\infty}\BE_z\sup_{0\le t\le T_\iota}
|M^{[u_\alpha]}_t -\int_0^t\sigma\nabla u(\BBX_r)\,dW_r|^2=0
\]
for q.e. $z\in E_{0,T}$. When combined with (\ref{eq6.6}) this
completes the proof.
\end{proof}

\begin{definition}
We say that a measurable function $u:E_{0,T}\rightarrow\BR$ is a
solution of (\ref{eq6.1}) if
\begin{enumerate}
\item[(a)]  $\BP_z(\int_0^{T_{\iota}}|f(\BBX_t,u(\BBX_t),0)|\,dt<\infty)=1$,
$\BE_z\sup_{0\le t\le T_\iota}
|\int_0^tf(\BBX_r,u(\BBX_r),0)\,dr|^2<\infty$ and
$\BE_z\int_0^{T_\iota}\|g(\BBX_t,u(\BBX_t),0)\|^2\,dt<\infty$ for
q.e. $z\in E_{0,T}$,

\item [(b)] $u$ is quasi-continuous, $u\in M^2(0,T;H^1_0(D))$
and for q.e. $z\in E_{0,T}$,
\end{enumerate}
\vspace{-5mm}
\begin{align}
\label{eq6.05} u(z)&=E_z\Big(\varphi(\BBX_{T_\iota})
+\int_0^{T_\iota}f(\BBX_t,u(\BBX_t),\sigma\nabla
u(\BBX_t))\,dt \\
&\qquad\qquad+\int_0^{T_\iota}g(\BBX_t,u(\BBX_t),\sigma\nabla
u(\BBX_t))\,d^\dagger\beta_t^{\iota}\Big). \nonumber
\end{align}
\end{definition}

\begin{definition}
A pair $(Y^z,Z^z)\in\SM^2\times M^2$ is a solution of the BDSDE
\begin{align}
\label{eq6.04} Y^z_t&=\varphi(\BBX_{T_{\iota}})+\int^{T_{\iota}}_t
f(\BBX_r,Y^z_r,Z^z_r)\,dr \\
&\quad+\int^{T_{\iota}}_t
g(\BBX_r,Y^z_r,Z^z_r)\,d^{\dagger}\beta^{\iota}_r
-\int^{T_{\iota}}_0Z^z_rdW_r,\quad t\in[0,T_{\iota}] \nonumber
\end{align}
if $\BP_z(\int_0^T(|f(\BBX_t,Y^z_t,0)|
+\|g(\BBX_t,Y^z_t,0)\|^2)\,dt<\infty)=1$ and (\ref{eq6.04}) is
satisfied $\BP_z$-a.s.
\end{definition}

In much the same way as in the proof of Lemma \ref{stw2.1} one can
show that if (D1)--(D5) are satisfied then for q.e. $z\in E_{0,T}$
the data
\[
(\xi,\bar f(t,y,e),\bar g(t,y,e))= (\varphi(\BBX_{T_\iota}),
f(\BBX_t,y,e),g(\BBX_t,y,e)),\quad t\in[0,T],y\in\BR,e\in\BR^d
\]
satisfy  assumptions $\mbox{\rm{(B1)--(B4)}}$ under the measure
$\BP_z$. Therefore by Corollary \ref{cor5.3} and Theorem
\ref{tw4.1}, for q.e. $z\in E_{0,T}$ there exists a unique
solution $(Y^z,Z^z)$ of (\ref{eq6.04}).

\begin{theorem}
\label{tw6.2} Assume that \mbox{\rm(D1)--(D5)} are satisfied. Then
there exists a unique solution $u$ of \mbox{\rm SPDE
(\ref{eq6.1})}. Moreover, for q.e. $z\in E_{0,T}$,
\[
Y^z_t=u(\BBX_t),\,\,t\in [0,T_\iota],\,\,\BP_z\mbox{-a.s.},\quad
\sigma\nabla u(\BBX)=Z^z,\,\,l^1\otimes\BP_z\mbox{-a.e. on }
[0,T_\iota]\times\Omega\times\Omega'
\]
where $(Y^z,Z^z)$ is a solution of \mbox{\rm(\ref{eq6.04})}.
\end{theorem}
\begin{proof}
The proof of uniqueness is similar to the proof of Proposition
\ref{stw3.4}, with obvious modifications in Remark \ref{rem3.8}.
Therefore we only show the existence of a solution. Let
$(Y^0,Z^0)=(0,0)$ and for $n\ge 0$ let $(Y^{n+1},Z^{n+1})$ be a
solution of the BDSDE
\begin{align*}
Y^{n+1}_t&=\varphi(\BBX_{T_\iota})
+\int_t^{T_\iota}f(\BBX_r,Y^{n+1},Z^{n}_r)\,dr
+\int_t^{T_\iota}g(\BBX_r,Y^{n+1}_r,Z^n_r)\,d^\dagger \beta_r \\
&\quad-\int_t^{T_\iota} Z^{n+1}_r\, dW_r,\quad t\in
[0,T_\iota],\quad \BP_z\mbox{-a.s.}
\end{align*}
By Proposition \ref{stw5.1},
\begin{equation}
\label{eq6.2} Y^{n}_t=u^{n}(\BBX_t),\quad t\in [0,T_\iota],\quad
\BP_z\mbox{-a.s.}
\end{equation}
and
\begin{equation}
\label{eq6.03} \sigma\nabla u^{n}(\BBX)=Z^{n},\quad
l^1\otimes\BP_z\mbox{-a.e. on }
[0,T_\iota]\times\Omega\times\Omega'
\end{equation}
for q.e. $z\in E_{0,T}$, where $u^0=0$ and $u^{n}$, $n\ge1$, is a
solution of the  SPDE
\[
\frac{\partial u^n}{\partial t}+A_t u^n =-f(t,x,u^n,\sigma\nabla
u^{n-1}) -g(t,x,u^n,\sigma\nabla
u^{n-1})\,d^{\dagger}\beta_t,\quad u^n(T)=\varphi.
\]
We know that for q.e. $z\in E_{0,T}$ there exist a unique solution
$(Y^z,Z^z)$ of (\ref{eq6.04}) and from the proof of Theorem
\ref{tw4.1} it follows that
\begin{equation}
\label{eq6.3}
\lim_{n\rightarrow\infty}\|(Y^n,Z^n)-(Y^z,Z^z)\|_{\beta,z}=0,
\end{equation}
where the norm $\|\cdot\|_{\beta,z}$ is defined by (\ref{eq.norm})
with $\mathbb{E}$ replaced by $\mathbb{E}_z$. Applying the
Burkholder-Davis-Gundy inequality and using standard argument we
get
\begin{equation}
\label{eq6.4} \lim_{n\rightarrow\infty}\BE_z\Big(\sup_{0\le t\le
T_\iota}|Y^n_t-Y^z_t|^2+\int_0^{T_\iota}|Z^n_t-Z^z_t|^2\,dt\Big)=0
\end{equation}
for q.e. $z\in E_{0,T}$. Let us put
$u(z)=\lim_{n\rightarrow\infty} u_n(z)$ for those $z\in E_{0,T}$
for which it exists in probability $\mathbb{P}_z$, and $u(z)=0$
otherwise. Then from (\ref{eq6.2}) and (\ref{eq6.4}) we conclude
that for q.e. $z\in E_{0,T}$,
\[
Y^z_t=u(\BBX_t),\quad t\in [0,T_\iota],\quad \BP_z\mbox{-a.s.}
\]
By (\ref{eq6.03}), for any $\nu\in\hat S_{00}$ (for the definition
of $\hat S_{00}$ see Remark \ref{rem3.1}) we have
\[
E\int_{E_{0,T}}(|\sigma\nabla(u_n-u_m)(z)|^2\wedge1) \hat
R^{0,T}\nu(z)\,dz =\BE_{\nu}
\int_0^{T_\iota}(|Z^n_t-Z^m_t|^2\wedge1)\,dt,
\]
which by (\ref{eq6.4}) converges to zero as
$n,m\rightarrow\infty$. Therefore using \cite[Theorem
2.5]{Trutnau} (see also the remark at the end of Section 4 in
\cite{O2}) and applying standard argument (see the proof of
\cite[Theorem 5.2.1]{FOT}) shows that there is a measurable
$v:\Omega\times E_{0,T}\rightarrow\BR$ and a subsequence, still
denoted by $n$, such that  for q.e. $z\in E_{0,T}$, $\sigma\nabla
u_n(z)\rightarrow v(z)$ in probability $P$. From this and
(\ref{eq6.4}) it follows that
\[
v(\BBX)=Z^z,\quad l^1\otimes\BP_{z}\mbox{-a.e. on } [0,T_\iota]
\times\Omega\times\Omega'
\]
for q.e. $z\in E_{0,T}$. Hence the pair $(u(\BBX),v(\BBX))$ is a
solution of (\ref{eq6.04}) for q.e $z\in E_{0,T}$. Consequently,
(\ref{eq6.05}) with $\sigma\nabla u$ replaced by $v$ is satisfied
for q.e. $z\in E_{0,T}$. Applying now Proposition \ref{stw5.1}
shows that $u\in M^2(0,T;H^1_0(D))$ and for q.e. $z\in E_{0,T}$,
\[
0= \BE_z\langle\int^{\cdot}_0(\sigma\nabla
u-v)(\BBX_t)\,dW_t\rangle_{T_{\iota}} =
\BE_z\int^{T_{\iota}}_0|(\sigma\nabla u-v)(\BBX_t)|^2\,dt,
\]
which completes the proof of the theorem.
\end{proof}

\section{Probabilistic and mild solutions of SPDEs}
\label{sec7}

In this section we adopt the following notation.
\begin{itemize}

\item $U$ is a separable real Hilbert space, $\{e_k\}_{k\ge1}$ is an
orthonormal basis of $U$, $Q$ is a symmetric nonnegative trace
class operator on $U$ such that $Qe_k=\lambda_ke_k$, $k\ge1$. We
will assume that $U\subset H=L^2(E;m)$. $U_0=Q^{1/2}(U)\subset U$
is the separable Hilbert space with the inner product $\langle
u,v\rangle_{U_0}=\langle Q^{-1/2}u,Q^{-1/2}v\rangle_U$ (Note that
$\{f_k=\sqrt{\lambda_k}e_k\}$ is an orthonormal basis of $U_0$).

\item $L(U_0,\BR)$ is the Banach space of all bounded operators
from $U_0$ into $\BR$ endowed with the supremum norm and
$L_2(U_0,\BR)$, $L_2(U_0,H)$ are the spaces of Hilbert-Schmidt
operators from $U_0$ into $\BR$ and $H$, respectively.

\item $B$ is a $Q$-Wiener process defined on some complete
probability space $(\Omega,\FF,P)$ with values in $U$.
\end{itemize}
It is known (see, e.g., \cite[Chapter 4]{DPZ}) that $B$ has the
expansion
\begin{equation}
\label{eq7.1} B_t=\sum^{\infty}_{k=1}\sqrt{\lambda_k}\beta^k_te_k
=\sum^{\infty}_{k=1}\beta^k_tf_k,\quad t\ge0,
\end{equation}
where $\beta^k_t=\lambda_k^{-1/2}\langle B_t,e_k\rangle_U$ are
real valued mutually independent standard Wiener processes on
$(\Omega,\FF,P)$ (the series above converges in
$L^2(\Omega,\FF,P;U)$ and $P$-a.s. in $C([0,T];U)$).

Suppose we are given $\tilde g:\Omega\times
E_{T}\times\BR\rightarrow\BR$ and $A_t$, $\varphi$, $f$  as in
Section \ref{sec3}. In this section we consider  SPDE of the form
\begin{equation}
\label{eq7.2} du(t)=-(A_tu+f(t,x,u))\,dt-\tilde
g(t,x,u)\,d^{\dagger}B_t, \quad u(T)=\varphi.
\end{equation}
In what follows
\begin{equation}
\label{eq7.3} g=(g_1,g_2,\dots),\quad g_k(t,x,y)=\tilde
g(t,x,y)\cdot f_k(x),\quad (t,x)\in E_{T}\,,y\in\BR.
\end{equation}

\subsection{Probabilistic solutions of (\ref{eq7.2})}
\label{sec7.1}

We assume that $\varphi,f,g$ satisfy (H1)--(H5). By $u$ we denote
the unique solution of SPDE (\ref{eq4.1})  of Theorem \ref{tw3.4}.

Let $t\mapsto G(\BBX_t,u(\BBX_t))$ be the process with values in
$L(U_0,\BR)$ defined as
\[
G(\BBX_t,u(\BBX_t))\psi=\tilde
g(\BBX_t,u(\BBX_t))\cdot\psi(\BBX_t),\quad t\ge0
\]
for $\psi\in U_0$. Since $g$ satisfies (H3b), (H4), (H5), and
$u(\BBX)$ is of class $M$, it follows from Remark \ref{rem3.1}
that the process $t\mapsto g_k(\BBX_t,u(\BBX_t))$ is of class $M$.
Since the family $\{f_k\otimes1,k\ge 1\}$, where $(f_k\otimes
1)(\psi)=\langle\psi,f_k\rangle_{U_0},\psi\in U_0$, is linearly
dense in $L_2(U_0,\BR)$ and
\begin{align*}
\langle G(t,\BBX_t,u(\BBX_t)),f_k\otimes
1\rangle_{L^2(U_0,\BR)}&=\sum^{\infty}_{i=1}
G(t,\BBX_t,u(\BBX_t))f_i\cdot f_k\otimes1(f_i)\\
&=\tilde g(t,\BBX_t,u(\BBX_t))\cdot f_k(\BBX_t) =
g_k(t,\BBX_t,u(\BBX_t)),
\end{align*}
it follows from \cite[Lemma 4.8]{DPZ} that $t\mapsto
G(\BBX_t,u(\BBX_t))$ is of class $M$. In particular, for every
$\omega'$ the process $t\mapsto G(\BBX_t,u(\BBX_t))$ is
$(\FF^{\beta^{\iota}}_{t,T_{\iota}})$-adapted.
 Moreover,
\begin{equation}
\label{eq7.12}
\BE_z\int_0^{T_\iota}|G(\BBX_t,u(\BBX_t))|^2_{L_2(U_0,\BR)}\,dt
=\BE_z\int_0^{T_{\iota}}\|g(\BBX_t,u(\BBX_r))\|^2\,dt<\infty,
\end{equation}
so for every $\omega'$, $t\mapsto G(\BBX_t,u(\BBX_t))$ is an
$(\FF^{\beta^{\iota}}_{t,T_{\iota}})$-adapted
$L_2(U_0;\BR)$-valued process. One can also check (see
\cite[Proposition 3.4]{DQS}) that if we set
\[
B^{\iota}_t=\sum^{\infty}_{k=1}\beta^k_{t+\iota(0)}f_k
=B_{t+\iota(0)},\quad t\in[0,T-\iota(0)],
\]
then
\begin{equation}
\label{eq7.13} \int^t_0G(\BBX_r,u(\BBX_r))\,d^{\dagger}B^{\iota}_r
=\int^t_0g(\BBX_r,u(\BBX_r))\,d^{\dagger}\beta^{\iota}_r
\end{equation}
for every $t\in[0,T]$. Therefore $u$ of Theorem \ref{tw3.4} is a
probabilistic solution of (\ref{eq7.2}) in the sense that for
q.e. $z\in E_{0,T}$,
\begin{equation}
\label{eq7.4} u(z)=E_z\Big(\varphi(\BBX_{T_\iota})
+\int_0^{T_\iota}f(\BBX_t,u(\BBX_t))\,dt
+\int_0^{T_\iota}G(\BBX_t,u(\BBX_t))\,d^\dagger B_t^{\iota}\Big)
,\quad P\mbox{-a.s.}
\end{equation}
By (\ref{eq3.17}), this can be written in the form
\begin{align}
\label{eq7.15} u(s,x)&=E_{s,x}\Big(\varphi(X_T)
+\int_0^{T-s}f(s+t,X_{s+t},u(s+t,X_{s+t}))\,dt\\
&\quad+\int_0^{T-s}G(s+t,X_{s+t},u(s+t,X_{s+t}))\,d^\dagger
B_t^{s}\Big),\quad P\mbox{-a.s.} \nonumber
\end{align}
\begin{remark}
\label{rem7.5} By \cite[(6.2.24)]{O3} and \cite[(4.4)]{O2}, if
$m(B)>0$ for some Borel set $B\subset E$ then
$\mbox{cap}(\{s\}\times B)>0$ for every $s\in\BR$. From this and
the fact that (\ref{eq7.15}) holds for q.e. $(s,x)\in E_{0,T}$ it
follows that for every $s\in(0,T]$ equality (\ref{eq7.15}) holds
true for $m$-a.e. $x\in E$.
\end{remark}

\subsection{Mild solutions of (\ref{eq7.2})}

In this subsection we assume additionally that $B^{(t)}=B^{(0)}$,
$t\in[0,T]$, and that $f$ is Lipschitz continuous with respect to
$y$, i.e. we assume (H1), (H2), (H3)(a), (H4), (H5) and the
following condition: there exists $L>0$ such that for every $z\in
E_{0,T}$,
\begin{equation}
\label{eq7.7} |f(z,y)-f(z,y')|\le L|y-y'|\quad\mbox{for }
y,y'\in\BR.
\end{equation}
Let $A$ denote the operator corresponding to the form $B^{(0)}$,
$\{P_t,t>0\}$ denote the semigroup of linear operators on $H$
associated with $B^{(0)}$, and let
\[
F:\Omega\times[0,T]\times H\rightarrow H, \quad
G:\Omega\times[0,T]\times H\rightarrow L_2(U_0,H)
\]
be operators defined as
\[
F(\omega,t,v)(x)=f(\omega,t,x,v(x)),\quad
(G(\omega,t,v)\psi)(x)=\tilde g(\omega,t,x,v(x))\cdot\psi(x)
\]
for $t\in[0,T]$, $v\in H$, $\psi\in U_0$, $x\in E$.  We shall show
that $u$ is a mild solution of equation (\ref{eq7.2}) interpreted
as abstract evolution equation of the form
\begin{equation}
\label{eq7.8} u(s)=\varphi+\int^T_s(Au(t)+F(t,u(t)))\,dt
+\int^T_sG(t,u(t))\,d^{\dagger}B_t,\quad u(T)=\varphi,
\end{equation}
on the Hilbert space $H$, i.e. if we set $u(s)(x)=u(s,x)$ for
$(s,x)\in E_{0,T}$ then
\begin{equation}
\label{eq7.5} u(s)=P_{T-s}\varphi+\int^{T}_sP_{t-s}F(t,u(t))\,dt
+\int^{T}_sP_{t-s}G(t,u(t))\,d^{\dagger}B_t.
\end{equation}
Here the integral involving $F$ is Bochner's integral. To see
this, let us first note that from (\ref{eq7.7}), (H3b) it follows
that $F(t,\cdot)$, $G(t,\cdot)$  are Lipschitz-continuous. Using
this and \cite[Lemma 4.8]{DPZ} one can show that $F$ (resp. $G$)
is an $(\FF^{\beta}_{t,T})$-progressively measurable mapping from
$\Omega\times[0,T] \times H$ into $(H,\BB(H))$ (resp. into
$(L^2(U_0,H),\BB(L^2(U_0,H)))$.
Since $\{P_t\}$ is a contraction on $H$, it follows from (H1),
(H3)(b) and (\ref{eq7.7}) that there is $c>0$ such that
\[
|P_tF(t,v)|_H\le c(1+|v|_H), \quad |P_tG(t,v)|_{L_2(U_0,H)}\le
c(1+|v|_H), \quad t\in[0,T].
\]
Since $u\in M^2(0,T;H)$, $E\int^{T}_s(|P_tF(t,u(t))|_H
+|P_tG(t,u(t))|^2_{L_2(U_0,H)})\,dt<\infty$, so the integrals in
(\ref{eq7.5}) involving $F$ and $G$ are well defined. By
\cite[Proposition A.2.2]{PR}, for $h\in H$ we have
\begin{align*}
\langle\int^{T}_sP_{t-s}F(t,u(t))\,dt,h\rangle_H&=\int^T_s\langle
P_{t-s}F(t,u(t)),h\rangle_H\,dt \\
&=\int^{T-s}_0\langle E_{s,\cdot}\,
f(s+t,X_{s+t},u(s+t,X_{s+t})),h\rangle_H\,dt\\
&=\langle\int^{T-s}_0
E_{s,\cdot}\,f(s+t,X_{s+t},u(s+t,X_{s+t}))\,dt,h\rangle_H.
\end{align*}
By the above and Fubini's theorem,
\begin{equation}
\label{eq7.11}
\langle\int^{T}_sP_{t-s}F(t,u(t))\,dt,h\rangle_H=\langle
E_{s,\cdot}\int^{T-s}_0
f(s+t,X_{s+t},u(s+t,X_{s+t})\,dt,h\rangle_H.
\end{equation}
Similarly, by \cite[Lemma 2.4.1]{PR} and the fact that the process
$X$ it time-homogeneous,
\begin{align*}
&\langle \int^T_sP_{t-s}G(t,u(t))\,d^{\dagger}B_t,h\rangle_H
=\sum^{\infty}_{k=1}\int^T_s\langle
P_{t-s}G(t,u(t))(f_k),h\rangle_H\,d^{\dagger}\beta^k_t\\
&\qquad=\sum^{\infty}_{k=1}\int^{T-s}_0\langle
E_{s,\cdot}\,g_k(s+t,X_{s+t},u(s+t,X_{s+t}),
h\rangle_H\,d^{\dagger}\beta^k_{t+s}\\
&\qquad=\sum^{\infty}_{k=1}\langle\int^{T-s}_0
E_{s,\cdot}\,g_k(s+t,X_{s+t},u(s+t,X_{s+t})
\,d^{\dagger}\beta^k_{t+s},h\rangle_H.
\end{align*}
By the above and the stochastic Fubini theorem (see \cite[Theorem
4.18]{DPZ}),
\begin{align}
\label{eq7.16}
&\langle\int^T_sP_{t-s}G(t,u(t))\,d^{\dagger}B_t,h\rangle_H\\
&\qquad=\sum^{\infty}_{k=1}\langle E_{s,\cdot}\int^{T-s}_0
g_k(s+t,X_{s+t},u(s+t,X_{s+t})
\,d^{\dagger}\beta^k_{t+s},h\rangle_H\nonumber\\
&\qquad= \langle E_{s,\cdot}\sum^{\infty}_{k=1}\int^{T-s}_0
g_k(s+t,X_{s+t},u(s+t,X_{s+t}))\,d^{\dagger}\beta^k_{s+t},h\rangle_H
\nonumber\\
&\qquad= \langle E_{s,\cdot}\int^{T-s}_0
G(s+t,X_{s+t},u(s+t,X_{s+t}))\,d^{\dagger}B^s_t,h\rangle_H.
\nonumber
\end{align}
From Remark \ref{rem7.5} and (\ref{eq7.11}), (\ref{eq7.16}) it
follows that $\langle u(s),h\rangle_H=\langle v(s),h\rangle_H$ for
$h\in H$, where $v(s)$ is defined by the right-hand side of
(\ref{eq7.5}). This shows that $u$ satisfies (\ref{eq7.5}). Thus
we have proved he following proposition.

\begin{proposition}
Assume that $B^{(t)}=B^{(0)}$for all $t\in[0,T]$, and
\mbox{\rm(H1)--(H5), (\ref{eq7.7})} are satisfied. Let $u$ be the
unique solution of SPDE \mbox{\rm(\ref{eq4.1})}. Then
\mbox{\rm(\ref{eq7.5})} holds true for every $s\in(0,T]$.
\end{proposition}

\begin{remark}
Let
\[
\hat B_t=B_{T-t}-B_T,\quad t\in[0,T].
\]
On can check  that
\[
\int^T_{T-s} P_{t-(T-s)}G(t,u(t))\,d^{\dagger}B_t =-\int^{s}_0
P_{s-t}G(T-t,u(T-t))\,d\hat B_t,\quad s\in[0,T].
\]
Therefore from (\ref{eq7.5}) it follows that $\bar u$ defined as
$\bar u(t)=u(T-t)$, $t\in[0,T]$, is a mild solution of the problem
\[
d\bar u(t)=(A\bar u+f(T-t,x,\bar u))\,dt-\tilde g(T-t,x,\bar
u)\,d\hat B_t, \quad \bar u(0)=\varphi.
\]
\end{remark}

\begin{remark}
\label{rem7.1}  Note that in case $A_t$ is defined by
(\ref{eq6.01}) one can generalize (\ref{eq7.4}) to equations of
the form (\ref{eq7.2}) with $f,\tilde g$ also depending on the
gradient of a solution (see (\ref{eq6.1})) and satisfying
(D1)--(D5). If, in addition, $A_t=A_0$, $t\in[0,T]$, and we
replace (D3)(a) by Lipschitz continuity in $y$, then in much the
same way as in the proof of (\ref{eq7.5}) one can prove  that $u$
of Theorem \ref{tw6.2} is a mild solution of (\ref{eq7.8}) with
the mappings $F:\Omega\times[0,T]\times H^1_0(D)\rightarrow H$,
$G:\Omega\times[0,T]\times H^1_0(D)\rightarrow L_2(U_0,H)$ defined
as
\[
F(\omega,t,v)(x)=f(\omega,t,x,v(x),\sigma\nabla v(x)),
\]
\[
(G(\omega,t,v)\psi)(x)=\tilde g(\omega,t,x,v(x),\sigma\nabla
v(x))\cdot\psi(x)
\]
for $t\in[0,T]$, $v\in H^1_0(D)$, $\psi\in U_0$, $x\in D$.
\end{remark}

\subsection{Examples}
\label{sec7.3}

Assume that $e_k\in L^{\infty}(E;m)$ and $e_k$ are bounded
uniformly in $k$, or, more generally,
\begin{equation}
\label{eq5.2} \sup_{x\in E}
\sum^{\infty}_{k=1}\lambda_k|e_k(x)|^2<\infty.
\end{equation}
Below we show that under (\ref{eq5.2}) the results of Sections
\ref{sec7.1} and \ref{eq7.2} apply to equation (\ref{eq7.2}) with
$\tilde g$ Lipschitz continuous  such that $\tilde g(\cdot,0)$ is
bounded or square integrable. For simplicity we assume that
$\tilde g$ does not depend on $\omega$.

\begin{example}
Assume that
\[
\tilde g(\cdot,0)\in L^2(E_{0,T};m_T)\cup
L^{\infty}(E_{0,T};m_T)
\]
and $\tilde g$ is Lipschitz continuous in $y$ with some constant
$L'$, i.e.
\[
|\tilde g(t,x,y_1)-\tilde g(t,x,y_2)|\le L'|y_1-y_2|
\]
for all $(t,x)\in E_{0,T}$ and $y_1,y_2\in\BR$. Then $g$ defined
by (\ref{eq7.3}) satisfies (H1) and (H3). Indeed, we have
\[
\|g_k(\cdot,0)\|^2_{L^2(E_{0,T};m_T)}=\lambda_k\|\tilde
g(\cdot,0)e_k(\cdot)\|^2_{L^2(E_{0,T};m_T)}
\]
and
\[
|g_k(t,x,y)-g_k(t,x,y')|\le L'\sqrt{\lambda_k}|e_k(x)|.
\]
By the last inequality, (H3) is satisfied with
$L_k(x)=L'\sqrt{\lambda_k}|e_k(x)|$. In case  $\tilde
g(\cdot,0)\in L^2(E_{0,T};m_T)$  assumption  (\ref{eq5.2})
immediately forces $g$ to satisfy (H1). If $\tilde g(\cdot,0)\in
L^{\infty}(E_{0,T};m_T)$ then $\|\tilde
g(\cdot,0)e_k(\cdot)\|^2_{L^2(E_{0,T};m_T)}\le T^2\|\tilde
g(\cdot,0)\|_{\infty}\|e_k\|_H$, so $g$ satisfies (H3), because
$\tr Q<\infty$ and we assume that $U\subset H$.
\end{example}

\begin{example}
Let $D$ be a nonempty bounded open subset of $\BR^d$ and let
$\tilde g:D_{0,T}\times\BR\times\BR^d\rightarrow\BR$ be a
measurable function. If $\tilde g(\cdot,0,0)\in
L^2(D_{0,T};m_T)\cup L^{\infty}(D_{0,T};m_T)$ and
\[
|\tilde g(t,x,y_1,e_1)-\tilde g(t,x,y_2,e_2)|\le
L'(|y_1-y_2|+|e_1-e_2|)
\]
for every $(t,x)\in E_{0,T}$ and $y_1,y_2\in\BR$,
$e_1,e_2\in\BR^d$ then $g=(g_1,g_2,\dots)$ defined by
$g_k(t,x,y,e)=\tilde g(t,x,y,e)\cdot f_k(x)$ satisfies (D1) and
(D3) with $L_k(x)=M_k(x)=L'\sqrt{\lambda_k}|e_k(x)|$, $x\in D$.
\end{example}

\noindent{\bf Acknowledgements}
\smallskip\\
This research was supported by  NCN Grant No. 2012/07/B/ST1/03508.

\vspace{3mm} {\small \noindent Tomasz Klimsiak: Institute of
Mathematics, Polish Academy of Sciences, \'Sniadeckich 8,
00-956 Warszawa, Poland\\
and\\
Faculty of Mathematics and Computer Science, Nicolaus Copernicus
University, Chopina 12/18, 87-100 Toru\'n, Poland\\
E-mail: tomas@mat.umk.pl
\medskip\\
Andrzej Rozkosz: Faculty of Mathematics and Computer Science,
Nicolaus Copernicus
University, Chopina 12/18, 87-100 Toru\'n, Poland\\
E-mail: rozkosz@mat.umk.pl}


\begin{thebibliography}{KR}


\bibitem{BGP}
V. Bally, I. Gy\"ongy and E. Pardoux, {\em White noise driven
parabolic SPDEs with measurable drift},  J. Funct. Anal. {\bf 120}
(1994),  484--510.

\bibitem{BM}
V. Bally and A. Matoussi, {\em Weak solutions for SPDEs and
backward doubly stochastic differential equations}, J. Theoret.
Probab. {\bf 14} (2001), 125--164.

\bibitem{BCM}
B. Boufoussi, J. Van Casteren and M. Mrhardy, {\em Generalized
backward doubly stochastic differential equations and SPDEs with
nonlinear Neumann boundary conditions}, Bernoulli {\bf 13} (2007),
423--446.

\bibitem{BDHPS}
Ph. Briand, B. Delyon, Y. Hu, E. Pardoux and L.  Stoica, $L^{p}$
{\em solutions of Backward Stochastic Differential Equations},
Stochastic Process. Appl. {\bf 108} (2003), 109--129.

\bibitem{DQS}
R.C. Dalang and L. Quer-Sardanyons, {\em Stochastic integrals for
spde's: a comparison}, Expo. Math. {\bf 29} (2011), 67--109.

\bibitem{DS}
L. Denis and L. Stoica, {\em A general analytical result for
non-linear SPDE's and applications}, Electron. J. Probab. {\bf 9}
(2004) No. 23, 674--709.

\bibitem{DPZ}
G. Da Prato and J. Zabczyk, {\em Stochastic Equations in Infinite
Dimensions}, Cambridge University Press, Cambridge, 1992.

\bibitem{FOT}
M. Fukushima, Y. Oshima and M. Takeda, {\em Dirichlet Forms and
Symmetric Markov Processes}, Walter de Gruyter, Berlin, 1994.

\bibitem{GP}
I. Gy\"ongy and E. Pardoux, {\em On the regularization effect of
space-time white noise on quasi-linear parabolic partial
differential equations}, Probab. Theory Related Fields {\bf 97}
(1993), 211--229.

\bibitem{K}
O. Kallenberg, {\em Foundations of modern probability. Second
edition}, Springer-Verlag, New York, 2002.

\bibitem{K:SPA}
T. Klimsiak, {\em Reflected BSDEs and the obstacle problem for
semilinear PDEs in divergence form}, Stochastic Process. Appl.
{\bf 122} (2012), 134--169

\bibitem{K:JFA}
T. Klimsiak, {\em Semi-Dirichlet forms, Feynman-Kac functionals
and the Cauchy problem for semilinear parabolic equations},  J.
Funct. Anal. {\bf 268} (2015), 1205--1240.

\bibitem{KR:JFA}
T. Klimsiak and A. Rozkosz,  {\em Dirichlet forms and semilinear
elliptic equations with measure data},  J. Funct. Anal. {\bf 265}
(2013), 890--925.

\bibitem{KR:CM}
T. Klimsiak and A. Rozkosz, {\em Semilinear elliptic equations
with measure data and quasi-regular Dirichlet forms}, Colloq.
Math. (2016), DOI: 10.4064/cm6466-10-2015.

\bibitem{Kr}
N.V. Krylov, {\em An analytic approach to SPDEs. Stochastic
partial differential equations: six perspectives}, in: Math.
Surveys Monogr., 64, Amer. Math. Soc., Providence, RI, 1999,
185--242.

\bibitem{MR}
Z. Ma and  Z. R\"ockner,  {\em Introduction to the Theory of
(Non-Symmetric) Dirichlet Forms}, Springer-Verlag, Berlin, 1992.

\bibitem{O1}
Y. Oshima, {\em Some Properties of Markov processes associated
with time dependent Dirichlet forms}, Osaka J. Math. {\bf 29}
(1992), 103--127.

\bibitem{O2}
Y. Oshima, {\em Time-dependent Dirichlet forms and related
stochastic calculus}, Infin. Dimens. Anal. Quantum Probab. Relat.
Top. {\bf 7} (2004), 281--316.

\bibitem{O3}
Y. Oshima, {\em Semi-Dirichlet Forms and Markov Processes}, Walter
de Gruyter, Berlin, 2013.

\bibitem{PP}
E. Pardoux and S. Peng, {\em Backward doubly stochastic
differential equations and systems of quasilinear SPDEs}, Probab.
Theory Related Fields {\bf 98} (1994), 209--227.

\bibitem{PZ1}
S. Peszat and J. Zabczyk, {\em Stochastic evolution equations with
a spatially homogeneous Wiener process}, Stochastic Process. Appl.
{\bf 72} (1997), 187--204.

\bibitem{PZ2}
S. Peszat and J. Zabczyk, {\em Nonlinear stochastic wave and heat
equations},  Probab. Theory Related Fields {\bf 116} (2000),
421--443.

\bibitem{Pierre}
M. Pierre, {\em Parabolic Capacity and Sobolev Spaces},  J. Math.
Anal. {\bf 14} (1983), 522--533.

\bibitem{PR}
C. Pr\'ev\^ot and M. R\"ockner, {\em A Concise Course on
Stochastic Partial Differential Equations}, Lecture Notes in Math.
{\bf 1905}, Springer, Berlin, 2007.

\bibitem{R}
B.L. Rozovskii, {\em Stochastic evolution systems},  Kluwer,
Dordrecht, 1990.

\bibitem{Sharpe}
M. Sharpe,  {\em General Theory of Markov Processes}. Academic
Press, New York, 1988.

\bibitem{Stannat}
W. Stannat, {\em The theory of generalized Dirichlet forms and its
applications  in analysis and stochastics}, Mem. Amer. Math. Soc.
142 (1999), no. 678.

\bibitem{Trutnau}
G. Trutnau, {\em Stochastic Calculus of Generalized Dirichlet
Forms and Applications to Stochastic Differential Equations in
Infinite Dimensions}, Osaka J. Math. {\bf 37} (2000), 315--343.

\bibitem{W}
J.B. Walsh,  {\em An introduction to stochastic partial
differential equations}, in: \'Ecole d'\'et\'e de probabilit\'es
de Saint-Flour, XIV--1984, Lecture Notes in Math. {\bf 1180},
Springer, Berlin, 1986, 265--439.

\bibitem{WZ}
Z. Wu and F. Zhang, {\em BDSDEs with locally monotone coefficients
and Sobolev solutions for SPDEs}, J. Differential Equations {\bf
251} (2011), 759--784.

\bibitem{ZZ}
Q. Zhang and H. Zhao, {\em Stationary solutions of SPDEs and
infinite horizon BDSDEs}, J. Funct. Anal. {\bf 252}, (2007)
171--219.

\end{thebibliography}
\end{document}